\documentclass{amsart}

\usepackage{fmmathmacros}

\usepackage{amsmath,amssymb}
\usepackage{tikz}

\usepackage[ruled,vlined]{algorithm2e}
\SetKwProg{Function}{Function}{}{}
\SetKwInOut{Input}{Input}
\SetKwInOut{Output}{Output}

\usepackage{hyperref}

\newcommand{\K}{\mathbf{K}}
\newcommand{\Norm}{\mathrm{Norm}}
\newcommand{\disc}{\mathrm{disc}}
\newcommand{\Covering}{\sc{Covering}}

\newtheorem{lemma}{Lemma}[section]
\newtheorem{theorem}[lemma]{Theorem}
\newtheorem{proposition}[lemma]{Proposition}

\theoremstyle{definition}
\newtheorem{definition}[lemma]{Definition}
\theoremstyle{remark}
\newtheorem{remark}{Remark}[section]

\begin{document}

\title{Division algorithms for norm-Euclidean real quadratic fields --
part I}
\author{Fran\c{c}ois MORAIN}
\thanks{LIX, CNRS, INRIA, \'Ecole Polytechnique, Institut
Polytechnique de Paris, Palaiseau, France}
\email{morain@lix.polytechnique.fr}

\date{\today}

\begin{abstract}
We give a Euclidean division algorithm for the real quadratic
fields $\Q(\sqrt{m})$ for $m \in \{2, 3, 6, 7, 11, 19\}$,
with the property that the norm of the remainder depends on the first
Euclidean minimum of the field.
In each case, we cover the square $[-1/2, 1/2] \times [-1/2, 1/2]$ with
hyperbolas and give a list of these, together with regions covered.
We mechanize the proofs as much as we can, using exact
computations, in order to be able to reproduce them.
\end{abstract}

\maketitle

\section{Introduction}

Let $\K/\Q = \K(\theta)$ be a degree $d$ number field with defining
polynomial $f(X)$, and discriminant $D_K$, $\Norm(x)$ is the norm of
$x \in \K$; we denote by $\OK$ the ring of integers of $\K$.

The {\em Euclidean minimum} of $\xi \in \K$ is given by
$$m_\K(\xi) = \mathrm{inf}\{\Norm(\xi - \gamma), \gamma \in \OK\}.$$
The {\em Euclidean minimum} of $\K$ is
$$M_1(\K) = \mathrm{sup}\{m_\K(\xi), \xi \in \K\}.$$
A {\em critical point} (if it exists) $\xi$ is such that $m_\K(\xi) =
M_1(\K)$. See \cite{Cerri2007,Lezowski2014} for
theoretical results as well as efficient algorithms to compute the
euclidean minimum for fields of rather large degree.
It is clear that $M_1(\K) < 1$ (resp. $>1$), $\OK$ is norm-euclidean
(resp. is not). If $M_1(\K)=1$, it depends (e.g., see~\cite{Cerri2006}).
We refer to the master piece \cite{Lemmermeyer1995} for the history of
the problem, and a quasi-complete list of the results in the field,
including several references to the work on Barnes and Swinnerton-Dyer
on this topic.

\begin{proposition}
$\OK$ is norm-Euclidean if and only if for every $\xi \in \K$,
there is $\gamma \in \OK$ s.t.
\begin{equation}\label{eucdef}
 |\Norm(\xi-\gamma)| < 1.
\end{equation}
\end{proposition}
Let us define classes of euclidean division algorithms: 1-division
stands for an algorithm producing $\gamma$ as in (\ref{eucdef}). This
type is not really satisfactory to get a useful bound for a gcd,
say. An $M$-division algorithm corresponds to producing $\gamma$ such
that $|\Norm(\xi-\gamma)| \leq M$ instead. As
examples, an $M_1$-division algorithm assures of a rapidly convergent gcd
sequence. Few $M_1$-algorithms are known, since the proof of
euclideanity does not always give a useful algorithm, with the
exception of some cyclotomic
fields~\cite{ScWi1995,CaSc2010,JoLaNgNa2021} (related to the finding
of explicit versions of higher reciprocity laws). In theory, if we
could have an explicit algorithm to compute $m_K(x)$, we would end up
with an {\em optimal} euclidean division algorithm.

The aim of this article is describe $M_1$-algorithms for the case
of norm-Euclidean real quadratic fields $\Q(\sqrt{m})$ with $m \not\equiv
1\bmod 4$. The case of $m \equiv 1 \bmod 4$ will be treated in
\cite{Morain2026c}. See also \cite{Morain2026a} for the case of imaginary
quadratic fields.

Section~\ref{sct-tools} describes exact computations, notably finding
the sign of an algebraic expression in certain cases, needed for our
proofs later on. This is used in Section~\ref{sct3} to deal with real
quadratic fields. The proofs are geometric in nature, covering some
square in the plane using hyperbolas defined by the norm. Note that we
did not search for a minimal number of coverings for given $m$. Our
focus is on exact computations without floating point numbers, leading
to reproducible proofs. It should be noted that our Maple program is
available on \url{https://gitlab.inria.fr/morain/euclid}, together
with the results obtained in each case in human readable format.

\section{Tools}
\label{sct-tools}

\subsection{Rounding}

We begin with an easy Lemma.
\begin{lemma}\label{half}
For all $a_0 \in \Q$, there exists $\varepsilon \in \{\pm 1\}$ and $x
\in \Z$ such that $a_0 = x + \varepsilon a$ with $0 \leq a \leq 1/2$.
\end{lemma}

\noindent
{\em Proof:} take $x = \left\lfloor a_0\right\rceil$, so that $|a_0-x|
\leq 1/2$ and $\varepsilon = sign(a_0-x)$.
$\Box$

\subsection{Computing the sign of a quadratic polynomial}

\begin{lemma}\label{signquad}
Let $P(X) = \alpha X^2 + \beta X + \gamma \in \Q[X]$ with $\alpha \neq
0$. The sign of $P(x)$ for real $x \in [x_{\min}, x_{\max}]$ is:
$$\left\{
 \begin{array}{cl}
 sign(\alpha) & \text{ if }\disc(P) < 0 (\text{case } -); \\
 sign(\alpha) & \text{ if }\disc(P) = 0 \text{case 0: but for } x_0 = -\beta/(2
 \alpha)\\
 & \text{ where it is }0\text{ in case } x_0 \text{ belongs to the
 interval}; \\
 \text{see proof} & \text{ if }\disc(P) > 0.
 \end{array}
\right.$$
\end{lemma}

\medskip
\noindent
{\em Proof:} in the last case, $P$ has two roots $x_{\pm} = (-\beta
\pm \sqrt{\disc(P)})/(2 \alpha)$. We can concentrate on the case
$\alpha > 0$, since the sign of $P(x)$ is the opposite of that $-P(x)$
in this case. There are several cases, depicted using Figure~\ref{case-ab}.

\begin{description}
 \item [case a] if $x_{\max} \leq x_{-}$ (resp. $x_{+} \leq x_{\min}$), $P(x)$
decreases (resp. increases) from $x_{\min}$ to
$x_{\max}$ and its values are positive. The former sub-case is labeled
$a_1$, the former $a_2$.

 \item [case b] $[x_{\min}, x_{\max}] \subset [x_{-}, x_{+}]$:
 $sign(P(x)) \leq 0$, see Figure~\ref{case-ab}.

 \item [case c] $[x_{-}, x_{+}] \subsetneq [x_{\min}, x_{\max}]$:
 $sign(P(x))$ is not constant on the interval:
  \begin{description}
    \item [case d] $x_{\min} < x_{-} < x_{\max} < x_{+}$: $P(x_{\min}) > 0$,
$P(x_{\max}) < 0$.

     \item [case e] $x_{-} < x_{\min} < x_{+} < x_{\max}$: $P(x_{\min}) < 0$,
$P(x_{\max}) > 0$.
  \end{description}
\end{description}

\begin{figure}[hbt]
\begin{tikzpicture}[>=latex]
\draw[domain=-2.4:1.4] plot (\x, {\x*\x+\x-1});
\draw[->] (-3, 0) -- (2, 0);
\draw[->] (0, -1.75) -- (0, 2.5);
\node at (-1.618, 0) {$\bullet$}; \node at (-1.2, 0.2) {$x_{-}$};
\node at (0.618, 0) {$\bullet$}; \node at (1, 0.2) {$x_{+}$};
\node at (-2, 1) {$\bullet$}; \node at (-2.5, 1) {$x_{\min}$};
\node at (-1.8, 0.44) {$\bullet$}; \node at (-2.3, 0.44) {$x_{\max}$};
\node at (1, 1) {$\bullet$}; \node at (1.5, 1) {$x_{\min}$};
\node at (1.25, 1.825) {$\bullet$}; \node at (1.9, 1.825) {$x_{\max}$};
\node at (0, -2) {case a)};
\end{tikzpicture}
\begin{tikzpicture}[>=latex]
\draw[domain=-2.4:1.4] plot (\x, {\x*\x+\x-1});
\draw[->] (-3, 0) -- (2, 0);
\draw[->] (0, -1.75) -- (0, 2.5);
\node at (-1.618, 0) {$\bullet$}; \node at (-1.2, 0.2) {$x_{-}$};
\node at (0.618, 0) {$\bullet$}; \node at (1, 0.2) {$x_{+}$};
\node at (-1, -1) {$\bullet$}; \node at (-1.5, -1) {$x_{\min}$};
\node at (0.2, -0.76) {$\bullet$}; \node at (0.8, -0.76) {$x_{\max}$};
\node at (0, -2) {case b)};
\end{tikzpicture}

\caption{Sign of a quadratic polynomial on $[x_{\min}, x_{\max}]$. \label{case-ab}}
\end{figure}
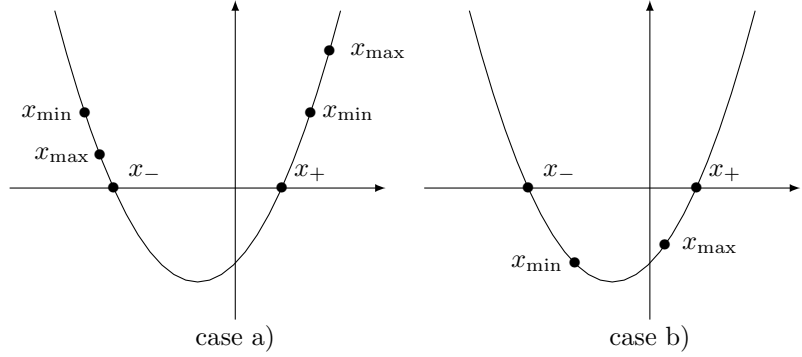

$\Box$

In our practice, the values of $x_{\pm}$, $x_{\min}$ and $x_{\max}$
are exact numbers, so that we can test for (in)equalities in an exact
way.

\subsection{Computing the sign of an exact number}

An {\em exact number} is either an integer, a rational number, a
square-root thereof, or some sum thereof. For instance
$$1, 1/2, \sqrt{2}, (1 + \sqrt{5})/2, \sqrt{1+\sqrt{2}}$$
are exact numbers. (In this work, all square-roots are supposed
positive.) Computations on these numbers are exact. For our
needs, we are be able to do sign computations, without using
floating point evaluations at all. This is the key to correct and
verifiable proofs.

\begin{proposition}
Let $t \geq 1$ be an integer.
Let $(x_i)_{0 \leq i \leq t}$ be strictly positive rational integers,
$(\epsilon_i)_{1\leq i\leq t}$ be signs $\pm 1$ and $(y_i)_{1 \leq i
\leq t}$ exact numbers. We want to determine the sign of the expression
$$x = x_0 + \sum_{i=1}^t \epsilon_i x_i \sqrt{y_i}.$$
Denote by $ABS$ the function
$$ABS(x) = x_0 + \sum_{i=1}^t x_i \sqrt{y_i}.$$

1) If $\epsilon_i = 1$ for all $i$, then $sign(x) = +1$.

2) If $t = 1$ and $\epsilon_1 = -1$, we have
$$sign(x) = sign( (x ABS(x)) = sign(x_0^2-x_1^2 y_1).$$

3) If $t=2$, then $sign(x) = sign(x ABS(x))$
and we decrease the number of radicals.

4) If $t=3$, then $sign(x) = sign(x ABS(x))$
and we decrease the number of radicals, only in the case where
$$x = x_0 + x_1 \sqrt{y_1} - x_2 \sqrt{y_2} - x_3 \sqrt{y_3}.$$
\end{proposition}

\medskip
\noindent
{\em Proof:} For 2), it is enough to remark that
$x_0 + x_1 \sqrt{y_1} > 0$ by assumption, therefore not changing the
sign of $x$ by multiplication.

3) When $t=2$, we reduce the problem to two cases (after reordering if
needed)
$$x = x_0 + x_1 \sqrt{y_1} - x_2 \sqrt{y_2}$$
or
$$x = x_0 - x_1 \sqrt{y_1} - x_2 \sqrt{y_2}.$$
In the first case, we use
$$sign(x) = sign(x ABS(x)) =
sign(x_0^2 + x_1^2 y_1-x_2^2 y_2+2 x_0 x_1 \sqrt{y_1})$$
and we are back to the case $t=1$.

In the second case, we do the same
$$sign(x) = sign(x ABS(x)) =
sign((x_0^2- x_1^2 y_1 - x_2^2 y_2) - 2 x_1 x_2 \sqrt{y_1 y_2})$$
and we have decreased the number of square-roots for the price of
increasing the content of the square-root. This sends us back to the case
$t=1$.

4) When $t=3$, we start with
$$x_0 + x_1 \sqrt{y_1} - x_2 \sqrt{y_2} - x_3 \sqrt{y_3}$$
which by multiplication yields:
$$x ABS(x) = (x_0+x_1\sqrt{y_1})^2-(x_2 \sqrt{y_2} + x_3 \sqrt{y_3})^2
= X_0 + X_1 \sqrt{Y_1} - X_2\sqrt{Y_2}. \Box$$
Numerical examples are given for $m=7$ in Section~\ref{ssct-7}.
\begin{remark}\label{rem2}
In some cases, we can use some tricks, for instance
write $x = x_0 + X_0 - Y$ with $X_0$ and $Y$ positive. If $X_0-Y$ is
easier to handle and $sign(X_0-Y) \geq 0$, then surely $x > 0$. If $X_0-Y
< 0$, we may have $x_0 \leq X_0$ and if $2X_0 -Y \leq 0$, then $x
\leq 0$. We could refine with $x_0 \leq x_1 \sqrt{y_1}$ and $X_0+x_1
\sqrt{y_1} \leq Y$; or even try all sums of subsets of $X$. There are
also cases where the number of square-roots does not decrease, but the
product $x ABS(x)$ yields the sign. For instance, when
$x = x_0 + x_1 \sqrt{y_1} + x_2 \sqrt{y_2} - x_3 \sqrt{y_3}$, we
compute
$$x ABS(x) 
= (x_0^2+x_1^2 y_1 + x_2^2 y_2 - x_3^2 y_3)
+2 (x_0 x_1 \sqrt{y_1} + x_0 x_2 \sqrt{y_2} + x_1 x_2 \sqrt{y_1
y_2})$$
and we get a positive sign if the first term is positive.
\end{remark}

\begin{remark}\label{rem1}
The preceding result can be easily extended to the case where
$y_k = P_k(a)$ for some polynomial in $\Q[X]$ that is positive on
the interval $\mathcal{A} = [a_{\min}, a_{\max}]$. This is the case
when we want to compare two pieces of hyperbolas in Section~\ref{sct3}.
\end{remark}

\section{Real quadratic fields $\Q(\sqrt{m})$, $m \not\equiv 1 \bmod
4$}
\label{sct3}

Let $m\not\equiv 1 \bmod 4$ be a square-free integer. Put
$\K = \Q(\sqrt{m})$; its ring of integers is $\OK = \Z[\omega]$ where
$\omega = \sqrt{m}$, and the defining polynomial of $\K$ is
$X^2-m$. The discriminant $D_K$ of $\K$ is equal to $4 m$ and
$$f_m(x, y):=\Norm(x+y \omega)= x^2-m y^2.$$
A number $x + y \omega$ for reals $x$, $y$ is identified with the
point $(x, y)$ in the plane.

The following Table recalls the list of the corresponding euclidean
number fields, together with minima and sets of critical points (some
of which coming from the very precious program of P.~Lezowski, see
\url{https://www.math.u-bordeaux.fr/~plezowsk/tables/result.php}).
$$\begin{array}{|r|c|c|}\hline
m & M_1 & C_1 \\ \hline
2 & 1/2 & \{(0, 1/2)\} \\
3 & 1/2 & \{(1/2, 1/2)\} \\
6 & 3/4 & \{(1/2, 1/2)\} \\
7 & 9/14 & \{(1/2, 5/14), (1/2, 9/14)\} \\
11 & 19/22 & \{(1/2, 15/22), (1/2, 7/22)\} \\
19 & 170/171 & \{(0, 20/57), (0, 37/57)\} \\
\hline
\end{array}$$

\subsection{Preparation}

\subsubsection{Hyperbolas and coverings}

We let $0 < M < 1$ be a real number (generally rational).
Since the norm form is $f_m(a, b) = a^2 - m b^2$, we consider
the hyperbolas of equations
$$H_{u, v}^{+}: (a+u)^2 - m (b+v)^2 = M, \quad H_{u, v}^{-}: (a+u)^2 - m
(b+v)^2 = -M$$
for $u$ and $v$ rational integers and note $\mathcal{H}_{u, v}$ the
region in between:
$$\mathcal{H}_{u, v} = \{(a, b) \in \mathcal{S}, -M \leq f_m(a+u,
b+v) \leq M\}.$$
For $(a, b) \in \mathcal{S}$, this forms 2 pieces, by symmetry
w.r.t. axes. Since $f_m(a+u, b+v) = f_m(-a-u,
-b-v)$ we reduce our study to $(a, b) \in \mathcal{S}_0 = [0,
1/2] \times [0, 1/2]$.

Consider the set $\mathcal{H}$ of $\mathcal{S}_0$ for
which $-M \leq f_m(a, b) \leq +M$. We say that $\mathcal{H}_{u, v}$
{\em covers} $(a, b)$ if $(a, b) \in \mathcal{H}_{u, v}$.
We extend this notion to that of a region $\mathcal{R} \subset
\mathcal{S}_0$. In short, we speak of $(u, v)$ covering a point
(resp. a region).

We end this list with a property related to critical points, which
tends to complicate the coverings, since they are {\em attracted} by
critical points.
\begin{proposition}
If $P = (x_0, y_0)$ is critical and is covered by $(u, v)$, then
$$f_m(x_0+u, y_0+v) = \pm M.$$
\end{proposition}

\subsubsection{Computations}

\begin{definition}\label{Buv}
For rational integer $u$ and $v$, and $\theta, \epsilon \in \{\pm
1\}$, define
$$B_{u, v}^{\theta, \epsilon}(a) := -v + \theta \sqrt{((a+u)^2 -
\epsilon M)/m}.$$
\end{definition}
Examples of plots for these curves are to be found in the numerous
cases below, with Figure~\ref{real2-3} to start with.
Note also that we need to draw pieces of our functions $B_{u,
v}^{\theta, \epsilon}(a)$ for $0 \leq a \leq 1/2$ and $u$ an integer
that can be relatively large. In that case we are happy to use the
approximation
$$B_{u, v}^{\theta, \epsilon}(a) \approx -v + \frac{\theta}{\sqrt{m}}
\, |a+u|.$$
\begin{lemma}\label{bfcta}
When $u \geq 0$, $B_{u, v}^{+, \mp}$ (resp. $B_{u, v}^{-, \mp}$) is
increasing (resp. decreasing) on $\mathcal{S}_0$; when $u < 0$, $B_{u,
v}^{+, \mp}$ (resp. $B_{u, v}^{-, \mp}$) is decreasing (resp. increasing) on
$\mathcal{S}_0$.
\end{lemma}

\noindent
{\em Proof:} since
$$B_{u, v}^{\theta, \epsilon}(a)' = \frac{\theta}{m}\,
\frac{a+u}{\sqrt{((a+u)^2 -\epsilon M)/m}},$$
the properties to be proven depend on the sign of $a+u$ on
$\mathcal{S}_0$, hence the result. $\Box$
Our tasks will be to prove
that $B_{u_1, v_1}^{\theta_1, \epsilon_1}(a) \geq B_{u_2,
v_2}^{\theta_2, \epsilon_2}(a)$ for $a \in \mathcal{A} = [a_{\min},
a_{\max}] \subset [0, 1/2]$. This will done using Remark~\ref{rem1},
as well as proving that $B_{u, v}^{\theta, \epsilon}(a) \geq b_0$ for
some fixed value $b_0$.

\subsection{Evaluating a sequence of norms}

We expand
$$f_m(a+u, b+v) = f_m(a, b) + f_m(u, v) + 2 (a u - m b v).$$
In some cases (computing $\mathcal{C}(p, B)$ below, testing, using it
in the actual division), $a$ and $b$ are fixed, but rational integers
$u$ and $v$ vary. We may precompute all $f_m(u, v)$ (and perhaps all
$m v$'s or at least $m b$).

\subsection{An informal algorithm}

Our goal is prove that for all $(a, b) \in \mathcal{S}_0$,
there exists a covering of $(a, b)$ by $\mathcal{H}_{u,
v}$ for some integers $u$, $v$. To help us, we
define for a point $p$:
$$\mathcal{C}(p, B) = \{(u, v)
\in \Z, |u|, |v| \leq B, p\text{ is covered by } H_{u, v}^{\pm}\},$$
for some integer $B$.
In practical computations, values $B = 10^2$ or $B = 10^3$. It is
clear that larger $B$'s do not produce a lot coverings.

We give as Algorithm~\ref{covering} a sketch of our approach. We
cannot prove that the algorithm terminates in all cases, nor say
anything on the complexity (but note this is a one-time
computation). Some parts of it are rather vague, but will be
demonstrated in the cases below. The idea of using the barycenter
relies on the idea to help cover some sub-region that could help
covering the whole region.

\begin{algorithm}[H]
\caption{Covering a region. \label{covering}}
  \SetKwInOut{Input}{input}\SetKwInOut{Output}{output}
  \SetKwProg{Fn}{Function}{}{}
\Fn{\Covering($m, M, \mathcal{R}$)}{
  \Input{$M$ is a bound on norms in $\Q(\sqrt{m})$}
  \Output{A covering for $\mathcal{R}$}
  \For{$p \in \mathcal{R}$}{
     $UV[p] = \mathcal{C}(p, B)$\;
  }
  \If{$\exists (u, v)$ common to all $UV[p]$}{
     \Return{$\mathcal{R}_{u, v}$ built from $\mathcal{R}$.}
  }
  \ElseIf{$\exists (u, v)$ covering more than two points}{
     build $\mathcal{R}_{u, v}$, the sub-region of $\mathcal{R}$
  covered by $(u, v)$, and its complement $\mathcal{R}'$\;
     \Return{$\mathcal{R}_{u, v} \cup \Covering(m, M, \mathcal{R}')$}
  }
  \Else{compute the barycentre $q$ of $\mathcal{R}$ and try to cover
  pieces of $\mathcal{R} \cup \{q\}$.
  }
}
\end{algorithm}
Our approach builds a sequence of points $P_i$, always starting
with $\mathcal{S}_0$ represented by its four corner points:
$$P_0 = (0, 0), P_1 = (1/2, 0), P_2 = (1/2, 1/2), P_3 = (0, 1/2).$$
Critical points will be denoted by $P_c$ or $P_{c_1}$, $P_{c_2}$,
\ldots; (such critical points will be printed with the symbol
$\circ$ in our figures). The output of this algorithm is a collection of
regions $\mathcal{R}_{u_i, v_i}$ that cover the square $\mathcal{S}_0$. A
typical regions is composed of (indices of) points related by lines or
pieces of hyperbolas $B_{u, v}^{\theta, \epsilon}$. See the case $m=6$
for examples of regions in Section~\ref{ssct-6}.

\subsection{The cases $m = 2, 3$}

The proofs of norm-Euclideanity for these three cases were done by
Perron~\cite{Perron1933}, $2$ and $3$ already done by
Dedekind~\cite{Dirichlet1893}. The case of 2 is carefully examined in
\cite{Varnavides1948ab}. For these cases, we proceed directly, but
looking at Figure~\ref{real2-3} is worthwhile.

\begin{proposition}
Let $0 \leq a, b \leq 1/2$.

i) For $m=2$, $|f_m(a, b)| \leq M(f_m) = 1/4$.

ii) For $m=3$, there exists $(u, v) \in \{(0, 0), (-1, 0)\}$ such that
$$|f_m(a+u, b+v)| \leq M(f_m).$$
\end{proposition}

\noindent
{\em Proof:}

i) For $0 \leq a, b \leq 1/2$, we get $-m/4 \leq f_m(a, b) \leq 1/4$, so that
$|f_m(a, b)| \leq m/4$. See Figure~\ref{real2-3}.

ii) For $m=3$, suppose that $a^2-3 b^2 < -1/2$. We deduce $3 b^2 >
1/2$, and with $-1 \leq a-1 \leq -1/2$ this implies
$$-\frac{1}{2} \leq \frac{1}{4} - 3 b^2 \leq (a-1)^2 - 3 b^2 < 1-\frac{1}{2}.$$
See Figure~\ref{real2-3}.
\begin{figure}[hbt]
\begin{tikzpicture}[>=latex,scale=1.15]
\draw[->] (-2.5, 0) -- (2.5, 0); \node at (2.4, 0.2) {$a$};
\draw[->] (0, -1.5) -- (0, 2); \node at (-0.2, 2) {$b$};
\draw[domain=-2:2] plot (\x, {sqrt((\x*\x+0.5)/2)});
\draw[domain=-2:2] plot (\x, {-sqrt((\x*\x+0.5)/2)});
\node at (0.75, 1.25) {$f_2 < -M_1$};
\node at (0.75, -1.25) {$f_2 < -M_1$};
\draw[domain=-2:-0.707] plot (\x, {sqrt((\x*\x-0.5)/2)});
\draw[domain=-2:-0.707] plot (\x, {-sqrt((\x*\x-0.5)/2)});
\draw[domain=0.7072:2] plot (\x, {sqrt((\x*\x-0.5)/2)});
\draw[domain=0.7072:2] plot (\x, {-sqrt((\x*\x-0.5)/2)});
\node at (-2, -0.5) {$f_2 > M_1$};
\node at ( 2, -0.5) {$f_2 > M_1$};
\draw[dashed] (-1, -1) -- (-1, 1);
\draw[dashed] ( 1, -1) -- ( 1, 1);
\draw[dashed] (-0.5, -1) -- (-0.5, 1);
\draw[dashed] ( 0.5, -1) -- ( 0.5, 1);
\draw[dashed] (-1, -0.5) -- (1, -0.5);
\draw[dashed] (-1, 0.5) -- (1, 0.5);
\node (P0) at (0, 0) {$\bullet$};
\node (Pc) at (0, 0.5) {$\circ$};
\end{tikzpicture}
\begin{tikzpicture}[>=latex,scale=1.15]
\draw[->] (-2.5, 0) -- (2.5, 0); \node at (2.4, 0.2) {$a$};
\draw[->] (0, -1.5) -- (0, 2); \node at (-0.2, 2) {$b$};
\draw[domain=-2:2] plot (\x, {sqrt((\x*\x+0.5)/3)});
\draw[domain=-2:2] plot (\x, {-sqrt((\x*\x+0.5)/3)});
\node at (0.75, 1.25) {$f_3 < -M_1$};
\node at (0.75, -1.25) {$f_3 < -M_1$};
\draw[domain=-2:-0.707] plot (\x, {sqrt((\x*\x-0.5)/3)});
\draw[domain=-2:-0.707] plot (\x, {-sqrt((\x*\x-0.5)/3)});
\draw[domain=0.7072:2] plot (\x, {sqrt((\x*\x-0.5)/3)});
\draw[domain=0.7072:2] plot (\x, {-sqrt((\x*\x-0.5)/3)});
\node at (-2, -0.5) {$f_3 > M_1$};
\node at ( 2, -0.5) {$f_3 > M_1$};
\draw[dashed] (-1, -1) -- (-1, 1);
\draw[dashed] ( 1, -1) -- ( 1, 1);
\draw[dashed] (-0.5, -1) -- (-0.5, 1);
\draw[dashed] ( 0.5, -1) -- ( 0.5, 1);
\draw[dashed] (-1, -0.5) -- (1, -0.5);
\draw[dashed] (-1, 0.5) -- (1, 0.5);
\node (P0) at (0, 0) {$\bullet$};
     \node at (-0.3, -0.3) {$P_0$};
\node (P1) at (0.5, 0) {$\bullet$};
     \node at (0.4, -0.2) {$P_1$};
\node (P2) at (0.5, 0.5) {$\circ$};
     \node at (0.3, 0.7) {$P_2$};
\node (P3) at (0, 0.5) {$\bullet$};
     \node at (-0.2, 0.7) {$P_3$};
\end{tikzpicture}
\caption{The cases $m=2$ (left), $m=3$ (right). \label{real2-3}}
\end{figure}
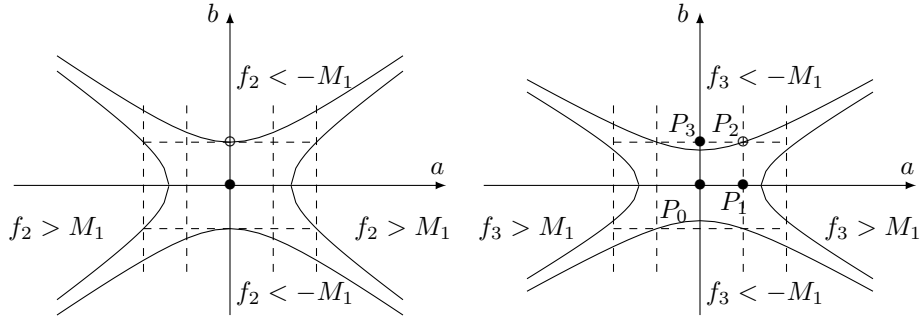

\subsection{The case $m=6$}
\label{ssct-6}

The proof of norm-Euclideanity for this case was done by
Perron~\cite{Perron1933}.

\begin{theorem}
For $m=6$, the square $\mathcal{S}_0$ can be covered by the pairs
$$(0, 0), (1, 0), (-2, -1).$$
\end{theorem}

\medskip
\noindent
{\em Proof:} 
There is one critical point in $\mathcal{S}_0$ which is $P_c = (1/2,
1/2)$. We start with the covering by $(0, 0)$, yielding the point
$$P_{4}=\left(0, 1/4\,\sqrt {2}\approx 0.353553\right).$$
The first region is
$$\mathcal{R}_{0, 0} = (0, 1, c, [0, 0, 1, -1], 4, 0),$$
which says that $(0, 0)$ covers the region delimited by
$\overline{P_0P_1P_c}$ followed by $B_{0, 0}^{+-}$ connecting $P_c$ to
$P_4$, followed by the line $\overline{P_4P_0}$.

We add covering by $(1, 0)$ which takes care of $P_{2}P_{3}P_{4}$,
introducing the point
$$P_{5}=B_{1, 0}^{++} \cap B_{0, 0}^{+-}=\left(1/4,
1/24\,\sqrt {78}\approx 0.367990\right).$$
We prove (using Remark~\ref{rem1}) the covering of $P_{2}P_{3}P_{4}$ with
$$B_{1, 0}^{+-} \geq B_{0, 0}^{+-}\text{ for } a \in [0, 1/2], \text{i.e.}, 
2\,a+1 \geq 0;$$
$$B_{0, 0}^{+-} \geq B_{1, 0}^{++}\text{ for } a \in [0, x(P_{5})],
\text{i.e.}, 1-4\,a \geq 0.$$
We prove that $P_2P_3$ is covered using
$$B_{1, 0}^{+-} \geq 1/2\text{ for } a \in [0, 1/2], \text{i.e.}, 
4\,{a}^{2}+8\,a+1 \geq 0: \text{case }a_2;$$ 
$$B_{1, 0}^{++} \leq 1/2\text{ for } a \in [0, 1/2], \text{i.e.}, 
4\,{a}^{2}+8\,a-5 \geq 0: \text{case b}.$$ 
This yields the region
$$\mathcal{R}_{1, 0} = (2, 3, 4, [0, 0, 1, -1], 5, [1, 0, 1, 1], 2).$$

Finally, we conclude with $(-2, -1)$ with proofs
$$B_{-2, -1}^{-+} \geq B_{0, 0}^{+-}\text{ for } a \in [x(P_{5}),
1/2], \text{i.e.},
-32\,{a}^{2}+112\,a-23 \geq 0: \text{case b};$$ 
$$B_{0, 0}^{+-} \geq B_{-2, -1}^{--}\text{ for } a \in [1/4, 1/2], \text{i.e.}, 
4\,{a}^{2}-8\,a+7 \geq 0: \text{case -}.$$ 
And the region is
$$\mathcal{R}_{-2, -1} = (c, 2, [-2, -1, -1, 1], 5, [0, 0, 1, -1], c).$$

\begin{figure}[hbt]
\begin{tikzpicture}[>=latex,scale=4]
\draw[->] (-0.1, 0) -- (1.5, 0); 
\draw[->] (0, -0.5) -- (0, 0.8); 
\draw[domain=-0.1:1] plot (\x, {sqrt((\x*\x+0.75)/6)}) node[right] {$B_{0, 0}^{+ -}$};
\draw[domain=-0.1:1] plot (\x, {-sqrt((\x*\x+0.75)/6)}) node[right] {$B_{0, 0}^{- -}$};
\draw[dashed] ( 0.5, -0.5) -- ( 0.5, 1);
\draw[dashed] ( 0, 0.5) -- ( 0.5, 0.5);

\coordinate (P_{0}) at (0.000000, 0.000000); \node at (P_{0}) {$\bullet$};
       \node at (-0.100000, -0.100000) {$P_{0}$};
\coordinate (P_{1}) at (0.500000, 0.000000); \node at (P_{1}) {$\bullet$};
       \node at (0.60000, -0.10000) {$P_{1}$};
\coordinate (P_{2}) at (0.500000, 0.500000); \node at (P_{2}) {$\bullet$};
       \node at (0.5500000, 0.600000) {$P_{2}$};
\coordinate (P_{3}) at (0.000000, 0.500000); \node at (P_{3}) {$\bullet$};
       \node at (-0.200000, 0.600000) {$P_{3}$};
\coordinate (P_{4}) at (0.000000, 0.353553); \node at (P_{4}) {$\bullet$};
       \node at (-0.1000000, 0.3) {$P_{4}$};
\coordinate (P_{c}) at (0.500000, 0.408248); \node at (P_{c}) {$\circ$};
       \node at (0.4500000, 0.345) {$P_{c}$};

\draw[domain=-.1339745962:1]
plot (\x, {(0.00000) * \x + (0.00000) + (1.00000) * sqrt((0.16667)*\x*\x+(0.33333)*\x+(0.04167))})
 node[below] {$B_{1, 0}^{+ +}$};

\draw[domain=-0.1:1]
plot (\x, {(0.00000) * \x + (0.00000) + (1.00000) * sqrt((0.16667)*\x*\x+(0.33333)*\x+(0.29167))})
 node[above] {$B_{1, 0}^{+ -}$};
\coordinate (P_{5}) at (0.250000, 0.367990); \node at (P_{5}) {$\bullet$};
     \node at (0.250000, 0.28) {$P_{5}$};

\draw[domain=-0.1:1,very thick]
plot (\x, {(0.00000) * \x + (1.00000) + (-1.00000) * sqrt((0.16667)*\x*\x+(-0.66667)*\x+(0.54167))})
 node[right] {$B_{-2, -1}^{- +}$};

\draw[domain=-0.1:1,very thick]
plot (\x, {(0.00000) * \x + (1.00000) + (-1.00000) * sqrt((0.16667)*\x*\x+(-0.66667)*\x+(0.79167))})
 node[below] {$B_{-2, -1}^{- -}$};

\end{tikzpicture}
\caption{The case $m=6$ (right). \label{real6}}
\end{figure}
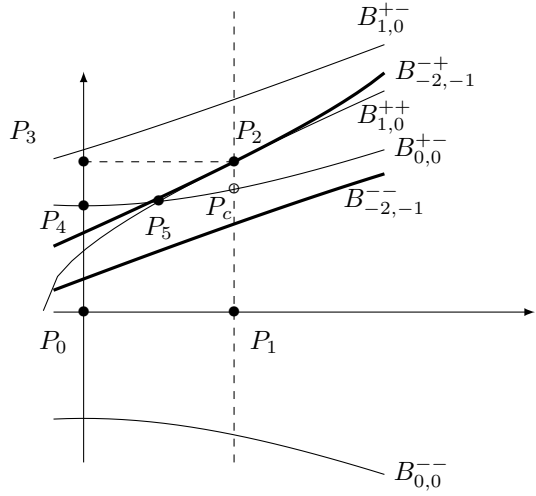

\subsection{The case $m=7$}
\label{ssct-7}

The first proof of Euclideanity was given by Perron~\cite{Perron1933}.
The first minimum was computed in~\cite{Varnavides1949}.

\begin{theorem}
For $m=7$, the square $\mathcal{S}_0$ can be covered by the four pairs
$$(0, 0), (1, 0), (-4, 1), (-2, -1).$$
\end{theorem}

\medskip
\noindent
{\em Proof:} 
There is one critical point in $\mathcal{S}_0$ which is $P_c = (1/2,
5/14)$, and we compute
$$\mathcal{C}(P_c, 10^2)=\{[-57, 21], [-26, -10], [-12,
4], [-4, 1], [-2, -1], [-1, 0], [0, 0], [1, -1],$$
$$[3, 1], [11, 4], [25, -10], [56, 21] \},$$

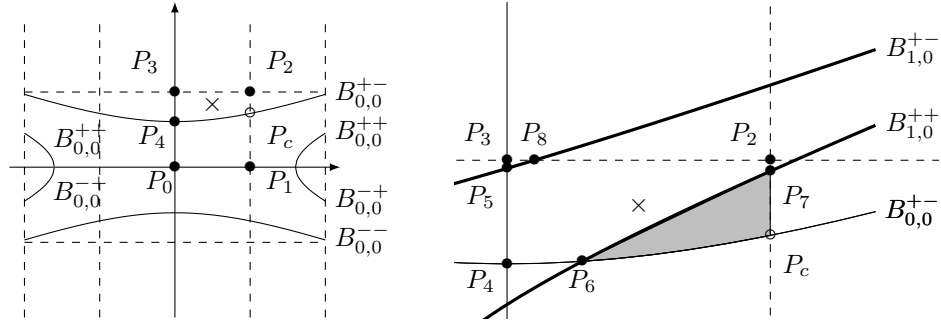
\begin{figure}[hbt]
\begin{center}
\begin{tikzpicture}[>=latex,xscale=2,yscale=2]
\draw[->] (-1.1, 0) -- (1.1, 0); 
\draw[->] (0, -1) -- (0, 1.1); 

\draw[domain=-1:1] plot (\x, {0+sqrt(((\x+(0))*(\x+(0))+0.6429)/7)}) node[right] {$B_{0, 0}^{+-}$};
\draw[domain=-1:1] plot (\x, {0-sqrt(((\x+(0))*(\x+(0))+0.6429)/7)}) node[right] {$B_{0, 0}^{--}$};
\draw[domain=-1:-.8017837257] plot (\x,
{0+sqrt(((\x+(0))*(\x+(0))-0.6429)/7)}) node[above] {$\quad\quad B_{0, 0}^{++}$};
\draw[domain=-1:-.8017837257] plot (\x,
{0-sqrt(((\x+(0))*(\x+(0))-0.6429)/7)}) node[below] {$\quad\quad B_{0, 0}^{-+}$};
\draw[domain=.8017837257:1] plot (\x, {0+sqrt(((\x+(0))*(\x+(0))-0.6429)/7)}) node[right] {$B_{0, 0}^{++}$};
\draw[domain=.8017837257:1] plot (\x, {0-sqrt(((\x+(0))*(\x+(0))-0.6429)/7)}) node[right] {$B_{0, 0}^{-+}$};

\draw[dashed] (-1, -1) -- (-1, 1);
\draw[dashed] ( 1, -1) -- ( 1, 1);
\draw[dashed] (-0.5, -1) -- (-0.5, 1);
\draw[dashed] ( 0.5, -1) -- ( 0.5, 1);
\draw[dashed] (-1, -0.5) -- (1, -0.5);
\draw[dashed] (-1, 0.5) -- (1, 0.5);
\node (P_4) at (0, 0.303) {$\bullet$}; \node at (-0.15, 0.2) {$P_4$};
\node (P_0) at (0, 0) {$\bullet$}; \node at (-0.1, -0.1) {$P_0$};
\node (P_1) at (0.5, 0) {$\bullet$}; \node at (0.7, -0.1) {$P_1$};
\node (P_2) at (0.5, 0.5) {$\bullet$}; \node at (0.7, 0.7) {$P_2$};
\node (P_3) at (0, 0.5) {$\bullet$}; \node at (-0.2, 0.7) {$P_3$};
\node at (0.5, 0.357) {$\circ$}; \node at (0.7, 0.2) {$P_c$};
\node (B) at (0.25, 0.4150) {$\times$};
\end{tikzpicture}
\hspace*{5mm}
\begin{tikzpicture}[>=latex,xscale=7,yscale=7]
\clip (-0.1, 0.2) rectangle (0.82, 0.8);
\draw[->] (-2, 0) -- (2, 0) node[very near end,above] {$a$};
\draw[->] (0, -1) -- (0, 1.5) node[very near end,left] {$b$};

\coordinate (P_{c}) at (0.5, 0.357);
\coordinate (P_{6}) at (0.142857, 0.307818);
\coordinate (P_{7}) at (0.500000, 0.479157);

\draw[domain=0.142857:0.5] plot (\x, {0+sqrt(((\x+(0))*(\x+(0))+0.6429)/7)})
(P_{c}) -- (P_{7}) -- (P_{6}) [fill=lightgray];
\draw[domain=-1:0.7] plot (\x, {0+sqrt(((\x+(0))*(\x+(0))+0.6429)/7)})
node[right] {$B_{0, 0}^{+-}$};

\draw[domain=-1:0.7] plot (\x, {0+sqrt(((\x+(0))*(\x+(0))+0.6429)/7)}) node[right] {$B_{0, 0}^{+-}$};
\draw[domain=-1:0.7] plot (\x, {0-sqrt(((\x+(0))*(\x+(0))+0.6429)/7)}) node[right] {$B_{0, 0}^{--}$};

\draw[domain=-1:0.7,very thick] plot (\x, {0+sqrt(((\x+(1))*(\x+(1))+0.6429)/7)}) node[right] {$B_{1, 0}^{+-}$};
\draw[domain=-1:0.7,very thick] plot (\x, {0-sqrt(((\x+(1))*(\x+(1))+0.6429)/7)}) node[right] {$B_{1, 0}^{--}$};
\draw[domain=-.1982162743:0.7,very thick] plot (\x, {0+sqrt(((\x+(1))*(\x+(1))-0.6429)/7)}) node[right] {$B_{1, 0}^{++}$};
\draw[domain=-.1982162743:0.7,very thick] plot (\x, {0-sqrt(((\x+(1))*(\x+(1))-0.6429)/7)}) node[right] {$B_{1, 0}^{-+}$};

\draw[dashed] (-1, -1) -- (-1, 1);
\draw[dashed] ( 1, -1) -- ( 1, 1);
\draw[dashed] (-0.5, -1) -- (-0.5, 1);
\draw[dashed] ( 0.5, -1) -- ( 0.5, 1);
\draw[dashed] (-1, -0.5) -- (1, -0.5);
\draw[dashed] (-1, 0.5) -- (1, 0.5);
\node (P_4) at (0, 0.303) {$\bullet$}; \node at (-0.05, 0.27) {$P_4$};
\node (P_2) at (0.5, 0.5) {$\bullet$}; \node at (0.45, 0.55) {$P_2$};
\node (P3) at (0, 0.5) {$\bullet$};
      \node[anchor=west] at (-0.1, 0.55) {$P_3$};
\coordinate (P_{5}) at (0.000000, 0.484452);
\node at (P_{5}) {$\bullet$};
       \node at (-0.050000, 0.43) {$P_{5}$};
\node at (P_{6}) {$\bullet$};
       \node at (0.142857, 0.27) {$P_{6}$};
\node at (P_{7}) {$\bullet$};
       \node at (0.550000, 0.43) {$P_{7}$};
\coordinate (P_{8}) at (0.052209, 0.500000);
\node at (P_{8}) {$\bullet$};
       \node at (0.052209, 0.55) {$P_{8}$};
\node at (0.5, 0.357) {$\circ$}; \node at (0.55, 0.3) {$P_c$};
\node (B) at (0.25, 0.4150) {$\times$};
\end{tikzpicture}
\end{center}
\caption{The case $m=7$: step 1 (left) and step 2 (right). \label{real7-step2}}
\end{figure}

Let us consider Figure~\ref{real7-step2}, with
$$P_{4} = \left(0, 3 \sqrt{2}/14 \approx 0.3030\right),$$
which appears naturally as the intersection of $B_{0, 0}^{+-}$ with the
$b$-axis. Points in-between the curves $B_{0, 0}^{+-}$ and
$B_{0, 0}^{--}$ are in $\mathcal{H}$. This gives us our first region:
$$\mathcal{R}_{0, 0} = (0, 1, c, [0, 0, 1, -1], 4, 0),$$
that is $(0, 0)$ covers $\overline{P_0P_1P_c} \cup B_{0, 0}^{+-} \cup
\overline{P_4P_0}$.

We have to find other regions
$\mathcal{H}_{u, v}$ that cover the remaining part of
$\mathcal{S}_0$, that is the region $\widehat{P_4 P_c} \cup
\overline{P_c P_2 P_3 P_4}$.

\medskip
\noindent
{\em Step 2:} First, we add the region $\mathcal{H}_{1, 0}$ that covers the
barycenter (plotted as $\times$) of the parallelogram $P_c P_2 P_3 P_4$, see
Figure~\ref{real7-step2}, which gives us a covering of the region
$\overline{P_5P_4}\cup \widehat{P_4P_6P_7} \cup
\overline{P_7P_2P_8}\cup \widehat{P_8 P_5}$ where
$$P_{5}=B_{1, 0}^{+-} \cap \{x=0\}=\left(0, 1/14\,\sqrt {46}\approx 0.484452\right),$$
$$P_{6}=B_{1, 0}^{++} \cap B_{0, 0}^{+-}=\left(1/7\approx 0.142857, {\frac {\sqrt {910}}{98}}\approx 0.307818\right),$$
$$P_{7}=B_{1, 0}^{++} \cap \{x=1/2\}=\left(1/2, 3/14\,\sqrt {6}\approx 0.479157\right),$$
$$P_{8}=B_{1, 0}^{+-} \cap \{y=1/2\}=\left(-1+1/14\,\sqrt {217}\approx 0.052209, 1/2\right).$$
It turns out that the arc of hyperbola $\widehat{P_4P_6}$ is also
covered. For this, using Lemma~\ref{bfcta} and Remark~\ref{rem1}, we
have to prove that for
$0 \leq a \leq 1/7$
$$B_{1, 0}^{+-} \geq B_{0, 0}^{+-}\text{ for } a \in [0, 1/7], \text{i.e.}, 
2\,a+1 \geq 0;$$
$$B_{0, 0}^{+-} \geq B_{1, 0}^{++}\text{ for } a \in [0, 1/7], \text{i.e.}, 
1-7\,a \geq 0.$$
We need to prove also that the line $P_8P_2$ is covered using
Remark~\ref{rem1}:
$$B_{1, 0}^{+-} \geq 1/2\text{ for } a \in [x(P_8), 1/2], \text{i.e.}, 
28\,{a}^{2}+56\,a-3 \geq 0: \text{case }a_2;$$ 
$$B_{1, 0}^{++} \leq 1/2\text{ for } a \in [x(P_8), 1/2], \text{i.e.}, 
28\,{a}^{2}+56\,a-39 \geq 0: \text{case b}.$$
This gives us the region
$$\mathcal{R}_{1, 0} =
(7, 2, 8, [1, 0, 1, -1], 5, 4, [0, 0, 1, -1], 6, [1, 0, 1, 1], 7).$$

\medskip
\noindent
{\em Step 3:} There remains to cover $\widehat{P_5P_8} \cup
\overline{P_8P_3P_5}$, which
is done by intersecting the
possible coverings of the points, yielding $\{(-4, 1), (-4, -2)\}$. We
select the first one to get Figure~\ref{real7-step3}.
To prove that $\mathcal{H}_{-4, 1}$
covers the arc $\widehat{P_4P_8}$, we need to prove:
$$B_{-4, 1}^{+-} \geq B_{1, 0}^{+-}\text{ for } a \in [0, x(P_8)],$$
In other words: $4 a^2-12 a+1 \geq 0,
\quad
-392 a^2+1036 a+5 \geq 0,
$
for $0 \leq a \leq x(P_8)$, which is true using
Lemma~\ref{signquad}.
We also need to prove covering of $P_3P_8$:
$$B_{-4, 1}^{+-} \geq 1/2\text{ for } a \in [0, x(P_8)], \text{i.e.}, 
28\,{a}^{2}-224\,a+25 \geq 0: \text{case }a_1;$$ 
$$B_{-4, 1}^{++} \leq 1/2\text{ for } a \in [0, x(P_8)], \text{i.e.}, 
28\,{a}^{2}-224\,a-11 \geq 0: \text{case b}.$$ 
The region is
$\mathcal{R}_{-4, 1} = (c, 5, [1, 0, 1, -1], 8, 3)$.

\medskip
We do the same
thing for the arc $\widehat{P_6P_cP_7}$ and find one pair $(-2, -1)$
that we add to Figure~\ref{real7-step3} and we prove first that
$$B_{-2, -1}^{-+} \geq B_{1, 0}^{++}\text{ for } a \in [0, 1/2], \text{i.e.}, 
4\,{a}^{2}-4\,a+3 \geq 0: \text{case -};$$
$$B_{1, 0}^{++} \geq B_{-2, -1}^{--}\text{ for } a \in [0, 1/2], \text{i.e.}, 
-392\,{a}^{2}+1148\,a+129 \geq 0: \text{case b}.$$
We also need
$$B_{0, 0}^{+-} \geq B_{-2, -1}^{--}\text{ for } a \in [0, 1/2], \text{i.e.}, 
4\,{a}^{2}-8\,a+3 \geq 0: \text{case }a_1,$$
which finishes the proof. The last region is
$$\mathcal{R}_{-2, -1} = (c, 7, [1, 0, 1, 1], 6, [0, 0, 1, -1], c).$$

$\Box$
 
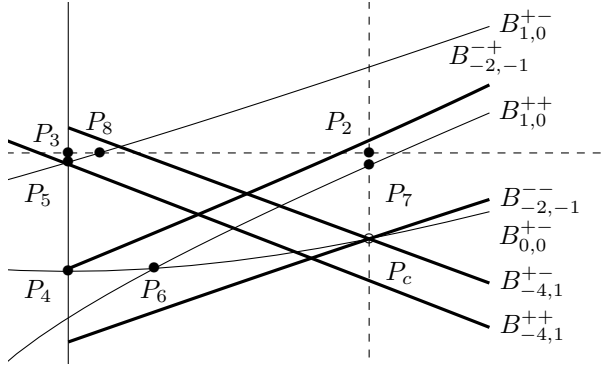
\begin{figure}[hbt]
\begin{center}
\begin{tikzpicture}[>=latex,xscale=8,yscale=8]
\clip (-0.1, 0.15) rectangle (0.9, 0.75);
\draw[->] (-2, 0) -- (2, 0) node[very near end,above] {$a$};
\draw[->] (0, -1) -- (0, 1.5) node[very near end,left] {$b$};
\coordinate (P_{c}) at (0.5, 0.357);

\draw[domain=-1:0.7] plot (\x, {sqrt((\x*\x+0.643)/7)}) node[below right] {$B_{0, 0}^{+-}$};
\draw[domain=-1:0.7] plot (\x, {-sqrt((\x*\x+0.643)/7)}) node[right] {$B_{0, 0}^{-+}$};
\draw[domain=-1.5:-0.802] plot (\x, {sqrt((\x*\x-0.643)/7)});
\draw[domain=-1.5:-0.802] plot (\x, {-sqrt((\x*\x-0.643)/7)});
\draw[domain=0.802:1.5] plot (\x, {sqrt((\x*\x-0.643)/7)});
\draw[domain=0.802:1.5] plot (\x, {-sqrt((\x*\x-0.643)/7)});

\draw[domain=-1:0.7] plot (\x, {sqrt(((\x+1)*(\x+1)+0.643)/7)})
node[right] {$B_{1, 0}^{+-}$};
\draw[domain=-0.198:0.7] plot (\x, {sqrt(((\x+1)*(\x+1)-0.643)/7)})
node[right] {$B_{1, 0}^{++}$};

\draw[domain=0:.7,very thick] plot (\x, {-1+sqrt(((\x+(-4))*(\x+(-4))+0.6429)/7)}) node[right] {$B_{-4, 1}^{+-}$};
\draw[domain=-1:0.7,very thick] plot (\x, {-1+sqrt(((\x+(-4))*(\x+(-4))-0.6429)/7)}) node[right] {$B_{-4, 1}^{++}$};
\draw[domain=-1:0.7,very thick] plot (\x, {-1-sqrt(((\x+(-4))*(\x+(-4))-0.6429)/7)}) node[right] {$B_{-4, 1}^{-+}$};

\draw[domain=0:.7,very thick] plot (\x, {1+sqrt(((\x+(-2))*(\x+(-2))+0.6429)/7)}) node[right] {$B_{-2, -1}^{+-}$};
\draw[domain=0:.7,very thick] plot (\x, {1-sqrt(((\x+(-2))*(\x+(-2))+0.6429)/7)}) node[right] {$B_{-2, -1}^{--}$};
\draw[domain=0:0.7,very thick] plot (\x, {1+sqrt(((\x+(-2))*(\x+(-2))-0.6429)/7)}) node[right] {$B_{-2, -1}^{++}$};
\draw[domain=0:0.7,very thick] plot (\x, {1-sqrt(((\x+(-2))*(\x+(-2))-0.6429)/7)}) node[above] {$B_{-2, -1}^{-+}$};

\draw[dashed] (-1, -1) -- (-1, 1);
\draw[dashed] ( 1, -1) -- ( 1, 1);
\draw[dashed] (-0.5, -1) -- (-0.5, 1);
\draw[dashed] ( 0.5, -1) -- ( 0.5, 1);
\draw[dashed] (-1, -0.5) -- (1, -0.5);
\draw[dashed] (-1, 0.5) -- (1, 0.5);
\node (P_4) at (0, 0.303) {$\bullet$}; \node at (-0.05, 0.27) {$P_4$};
\node (P_2) at (0.5, 0.5) {$\bullet$}; \node at (0.45, 0.55) {$P_2$};
\node (P3) at (0, 0.5) {$\bullet$}; \node[anchor=west] at (-0.075, 0.53) {$P_3$};
\node (P4) at (0, 0.484) {$\bullet$}; \node at (-0.05, 0.43) {$P_5$};
\node (P5) at (0.1428, 0.3078) {$\bullet$}; \node at (0.1428, 0.27) {$P_6$};
\node (P6) at (0.5, 0.479) {$\bullet$}; \node at (0.55, 0.43) {$P_7$};
\node (P7) at (0.0522, 0.5) {$\bullet$}; \node at (0.0522, 0.55) {$P_8$};
\node at (P_{c}) {$\circ$}; \node at (0.55, 0.3) {$P_c$};
\end{tikzpicture}
\end{center}
\caption{The case $m=7$: step 3. \label{real7-step3}}
\end{figure}

\subsection{The case $m=11$}

This case was done in \cite{Perron1933,Oppenheim1934,Remak1934}
and dealt with again in \cite[Theorem 1]{BaSw1952b} (where a sequence
of successive minima is given, together with their critical points and
many properties) and \cite{Varnavides1952}.

For the last two cases, we do not give regions, as they are
numerous. See the author git site for the corresponding data files.
\begin{theorem}
For $m=11$, the square $\mathcal{S}_0$ can be covered by the six pairs
$$(0, 0), (1, 0), (-6, -2), (-5, 1), (5, -2), (25, -8).$$
\end{theorem}

\medskip
\noindent
{\em Proof:} 
There is one critical point in $\mathcal{S}_0$ which is $P_c = (1/2,
7/22)$. We compute
$$\mathcal{C}(P_c, 10^2) = \{
[-16, -5], [-6, -2], [-1, 0], [0, 0], [5, -2], [15, -5] \}.$$

\begin{figure}[hbt]
\begin{center}
\begin{tikzpicture}[>=latex,xscale=2,yscale=2]
\draw[->] (-1.1, 0) -- (1.1, 0); 
\draw[->] (0, -0.5) -- (0, 1.1); 

\draw[domain=-1:1]
plot (\x, {(0.00000) * \x + (0.00000) + (1.00000) * sqrt((0.09091)*\x*\x+(0.00000)*\x+(0.07851))})
 node[right] {$B_{0, 0}^{+ -}$};

\draw[domain=-1:1]
plot (\x, {(0.00000) * \x + (0.00000) + (-1.00000) * sqrt((0.09091)*\x*\x+(0.00000)*\x+(0.07851))})
 node[right] {$B_{0, 0}^{- -}$};

\draw[dashed] (-1, -0.5) -- (-1, 0.5);
\draw[dashed] ( 1, -0.5) -- ( 1, 0.5);
\draw[dashed] (-0.5, -0.5) -- (-0.5, 0.5);
\draw[dashed] ( 0.5, -0.5) -- ( 0.5, 0.5);
\draw[dashed] (-1, -0.5) -- (1, -0.5);
\draw[dashed] (-1, 0.5) -- (1, 0.5);
\node (P_{4}) at (0, 0.2802) {$\bullet$}; \node at (-0.15, 0.2) {$P_{4}$};
\node (P_0) at (0, 0) {$\bullet$}; \node at (-0.1, -0.1) {$P_0$};
\node (P_1) at (0.5, 0) {$\bullet$}; \node at (0.7, -0.1) {$P_1$};
\node (P_2) at (0.5, 0.5) {$\bullet$}; \node at (0.7, 0.7) {$P_2$};
\node (P_3) at (0, 0.5) {$\bullet$}; \node at (-0.2, 0.7) {$P_3$};
\node (P1) at (0.5, 0.3182) {$\circ$}; \node at (0.7, 0.2) {$P_c$};
\node (B) at (0.25, 0.3996) {$\times$};
\end{tikzpicture}
\begin{tikzpicture}[>=latex,xscale=8,yscale=8]
\clip (-0.1, 0.05) rectangle (0.8, 0.7);
\draw[->] (-2, 0) -- (2, 0); 
\draw[->] (0, -1) -- (0, 1.5); 

\coordinate (P3) at (0.000000, 0.500000);
       \node at (-0.05, 0.55) {$P_{3}$};
\coordinate (P5) at (0.000000, 0.411608);
       \node at (-0.05, 0.45) {$P_{5}$};
\coordinate (P6) at (0.363636, 0.300888);
\node at (P6) {$\bullet$};
       \node at (0.3636, 0.25) {$P_{6}$};
\coordinate (P7) at (0.500000, 0.355011);
\node at (P7) {$\bullet$};
       \node at (0.53, 0.4) {$P_{7}$};
\coordinate (P8) at (0.373450, 0.500000);
\node at (P8) {$\bullet$};
       \node at (0.3734, 0.55) {$P_{8}$};
\coordinate (Pc) at (0.500000, 0.318182);

\draw[fill=lightgray] (P3) -- (P5) -- (P8) -- (P3);
\draw[fill=lightgray] (Pc) -- (P7) -- (P6) -- (Pc);

\node at (P3) {$\bullet$};
\node at (P5) {$\bullet$};
\node at (Pc) {$\circ$};
       \node at (0.54, 0.27) {$P_c$};

\draw[domain=-1:0.6] plot (\x, {sqrt((\x*\x+0.8636)/11)})
 node[right] {$B_{0, 0}^{+-}$};
\draw[domain=-1:0.7] plot (\x, {-sqrt((\x*\x+0.8636)/11)});
\draw[domain=-1.5:-0.930] plot (\x, {sqrt((\x*\x-0.8636)/11)});
 \node[anchor=west] at (1.55, 0.5) {$B_{0, 0}^{+-}$};
\draw[domain=-1.5:-0.930] plot (\x, {-sqrt((\x*\x-0.8636)/11)});
\draw[domain=0.930:1.5] plot (\x, {sqrt((\x*\x-0.8636)/11)});
\draw[domain=0.930:1.5] plot (\x, {-sqrt((\x*\x-0.8636)/11)});

\draw[domain=0:0.6,very thick]
plot (\x, {(0.00000) * \x + (0.00000) + (1.00000) * sqrt((0.09091)*\x*\x+(0.18182)*\x+(0.01240))})
 node[right] {$B_{1, 0}^{+ +}$};

\draw[domain=0:0.6,very thick]
plot (\x, {(0.00000) * \x + (0.00000) + (1.00000) * sqrt((0.09091)*\x*\x+(0.18182)*\x+(0.16942))})
 node[right] {$B_{1, 0}^{+ -}$};

\draw[domain=0:0.6,very thick]
plot (\x, {(0.00000) * \x + (0.00000) + (-1.00000) * sqrt((0.09091)*\x*\x+(0.18182)*\x+(0.01240))})
 node[right] {$B_{1, 0}^{- +}$};

\draw[domain=0:0.6,very thick]
plot (\x, {(0.00000) * \x + (0.00000) + (-1.00000) * sqrt((0.09091)*\x*\x+(0.18182)*\x+(0.16942))})
 node[right] {$B_{1, 0}^{- -}$};

\draw[dashed] ( 0.5, 0) -- ( 0.5, 0.5);
\draw[dashed] (0, 0.5) -- (0.5, 0.5);
\node (P0) at (0, 0.2802) {$\bullet$}; \node at (-0.05, 0.25) {$P_{4}$};
\node (P_2) at (0.5, 0.5) {$\bullet$}; \node at (0.55, 0.5) {$P_2$};
\node (B) at (0.25, 0.3996) {$\times$};
\end{tikzpicture}
\end{center}
\caption{The case $m=11$: step 1 (left), step 2 (right). \label{real11-step2}}
\end{figure}
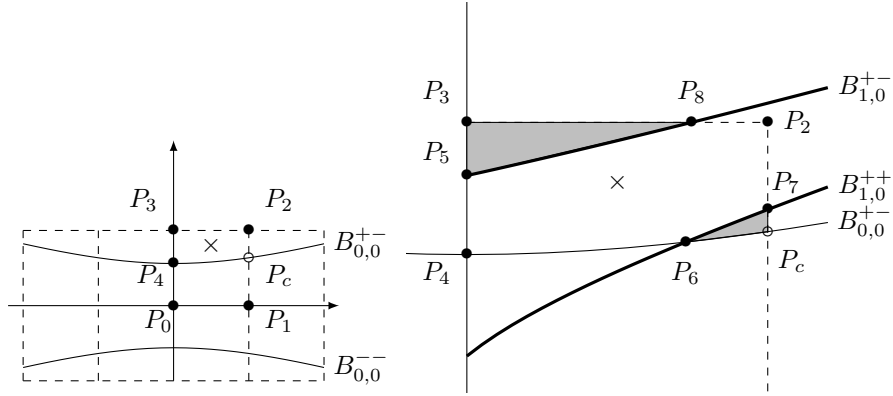

Let us consider Figure~\ref{real11-step2}, with
$P_{4} = (0, 1/22\,\sqrt {38})$.
Points in-between the curves $B_{0, 0}^{+-}$ and
$B_{0, 0}^{--}$ are in $\mathcal{H}$. We have to find other regions
$\mathcal{H}_{u, v}$ that cover the remaining parts of
$\mathcal{S}_0$, namely the region $\widehat{P_{4} P_c} \cup
\overline{P_c P_2 P_3 P_{4}}$.

\medskip
\noindent
{\em Step 2:} 
we add the region $\mathcal{H}_{1, 0}$ that covers the
barycenter of the parallelogram $\overline{P_{4} P_c P_2 P_3}$, see
Figure~\ref{real11-step2}, which gives us a covering of
$\overline{P_{5}P_{4}} \cup \widehat{P_{4}P_{6}P_{7}} \cup \overline{P_{7}P_2P_{8}}
\cup \widehat{P_{8} P_{5}}$ where
$$P_{5} = B_{1, 0}^{+-} \cap \{x=0\} = (0, 1/22\,\sqrt {82}\approx 0.4116),$$
$$P_{6} = B_{1, 0}^{++} \cap B_{0, 0}^{+-}=\left(4/11, {\frac {\sqrt {5302}}{242}}\approx 0.3009\right),$$
$$P_{7} = B_{1, 0}^{++} \cap \{x=1/2\} = \left(1/2, 1/22\,\sqrt
{61}\approx 0.3550\right),$$
$$P_{8} = B_{1, 0}^{+-} \cap \{y=1/2\}=\left(-1+1/22\,\sqrt {913}\approx 0.3734, 1/2\right).$$
We have to prove that the arc of hyperbola $\widehat{P_{4}P_{6}}$ is also
covered. For this, using Lemma~\ref{bfcta} and Remark~\ref{rem1}, we
prove that
$$B_{1, 0}^{+-} \geq B_{0, 0}^{+-}\text{ for } a \in [0, 1/2], \text{i.e.}, 
2\,a+1 \geq 0;$$
$$B_{0, 0}^{+-} \geq B_{1, 0}^{++}\text{ for } a \in [0, x(P_{6})],
\text{i.e.}, 4-11\,a \geq 0.$$
We also need to prove the covering of $P_8P_2$:
$$B_{1, 0}^{+-} \geq 1/2\text{ for } a \in [x(P_8), 1/2], \text{i.e.}, 
44\,{a}^{2}+88\,a-39 \geq 0: \text{case }a_2;$$ 
$$B_{1, 0}^{++} \leq 1/2\text{ for } a \in [x(P_8), 1/2], \text{i.e.}, 
44\,{a}^{2}+88\,a-115 \geq 0: \text{case b}.$$ 

\medskip
\noindent
{\em Step 3:} to cover $\widehat{P_{6}P_cP_{7}}$, we intersect the
possible coverings of the points, yielding $(-6, -2)$ and
Figure~\ref{real11-step3}.
\begin{figure}[hbt]
\begin{center}
\begin{tikzpicture}[>=latex,xscale=14,yscale=14]
\clip (-0.1, 0.1) rectangle (0.8, 0.68);
\draw[->] (-2, 0) -- (2, 0) node[very near end,above] {$a$};
\draw[->] (0, -1) -- (0, 1.5) node[very near end,left] {$b$};

\coordinate (P3) at (0.000000, 0.500000);
\node at (P3) {$\bullet$};
       \node at (-0.05, 0.55) {$P_{3}$};
\coordinate (P5) at (0.000000, 0.411608);
\node at (P5) {$\bullet$};
       \node at (-0.05, 0.4) {$P_{5}$};
\coordinate (P_{9}) at (0.099043, 0.500000);
\node at (P_{9}) {$\bullet$};
       \node at (0.099043, 0.530000) {$P_{9}$};
\coordinate (P_{10}) at (0.000000, 0.466030);
\node at (P_{10}) {$\bullet$};
       \node at (-0.050000, 0.48) {$P_{10}$};

\draw[fill=lightgray] (P3) -- (P_{9}) -- (P_{10}) -- (P3);

\draw[domain=-1:0.7] plot (\x, {sqrt((\x*\x+0.8636)/11)})
 node[below] {$B_{0, 0}^{+-}$};
\draw[domain=-1:0.7] plot (\x, {-sqrt((\x*\x+0.8636)/11)});
\draw[domain=-1.5:-0.930] plot (\x, {sqrt((\x*\x-0.8636)/11)});
 \node[anchor=west] at (1.55, 0.5) {$B_{0, 0}^{+-}$};
\draw[domain=-1.5:-0.930] plot (\x, {-sqrt((\x*\x-0.8636)/11)});
\draw[domain=0.930:1.5] plot (\x, {sqrt((\x*\x-0.8636)/11)});
\draw[domain=0.930:1.5] plot (\x, {-sqrt((\x*\x-0.8636)/11)});

\draw[domain=0:0.7]
plot (\x, {(0.00000) * \x + (0.00000) + (1.00000) * sqrt((0.09091)*\x*\x+(0.18182)*\x+(0.01240))})
 node[right] {$B_{1, 0}^{+ +}$};

\draw[domain=0:0.6]
plot (\x, {(0.00000) * \x + (0.00000) + (1.00000) * sqrt((0.09091)*\x*\x+(0.18182)*\x+(0.16942))})
 node[right] {$B_{1, 0}^{+ -}$};

\draw[domain=0:0.6]
plot (\x, {(0.00000) * \x + (0.00000) + (-1.00000) * sqrt((0.09091)*\x*\x+(0.18182)*\x+(0.01240))})
 node[right] {$B_{1, 0}^{- +}$};

\draw[domain=0:0.6]
plot (\x, {(0.00000) * \x + (0.00000) + (-1.00000) * sqrt((0.09091)*\x*\x+(0.18182)*\x+(0.16942))})
 node[right] {$B_{1, 0}^{- -}$};

\draw[domain=-.1:.6,very thick]
plot (\x, {(0.00000) * \x + (2.00000) + (-1.00000) * sqrt((0.09091)*\x*\x+(-1.09091)*\x+(3.19421))})
 node[above] {$B_{-6, -2}^{- +}$};

\draw[domain=-.1:.7,very thick]
plot (\x, {(0.00000) * \x + (2.00000) + (-1.00000) * sqrt((0.09091)*\x*\x+(-1.09091)*\x+(3.35124))})
 node[right] {$B_{-6, -2}^{- -}$};

\draw[domain=-.1:.5,very thick]
plot (\x, {(0.00000) * \x + (1.00000) + (-1.00000) * sqrt((0.09091)*\x*\x+(-0.36364)*\x+(0.28512))})
 node[right] {$B_{-2, -1}^{- +}$};

\draw[domain=-.1:.6,very thick]
plot (\x, {(0.00000) * \x + (1.00000) + (-1.00000) * sqrt((0.09091)*\x*\x+(-0.36364)*\x+(0.44215))})
 node[right] {$B_{-2, -1}^{- -}$};

\draw[dashed] ( 0.5, 0) -- ( 0.5, 0.5);
\draw[dashed] (-1, -0.5) -- (1, -0.5);
\draw[dashed] (0, 0.5) -- (0.5, 0.5);
\node (P0) at (0, 0.2802) {$\bullet$}; \node at (-0.05, 0.25) {$P_{4}$};
\node (P_2) at (0.5, 0.5) {$\bullet$};
      \node at (0.53, 0.5) {$P_2$};
\node (P1) at (0.5, 0.3182) {$\circ$}; \node at (0.53, 0.28) {$P_c$};
\node (P_{6}) at (0.3636, 0.3009) {$\bullet$}; \node at (0.3636, 0.25) {$P_{6}$};
\node (P_{7}) at (0.5000, 0.3550) {$\bullet$};
        \node at (0.53, 0.35) {$P_{7}$};
\node (P_{8}) at (0.3734, 0.5000) {$\bullet$}; \node at (0.3734, 0.55) {$P_{8}$};
\node (B_2) at (0.1245, 0.4705) {$\times$}; 
\end{tikzpicture}
\end{center}
\caption{The case $m=11$: step 3. \label{real11-step3}}
\end{figure}
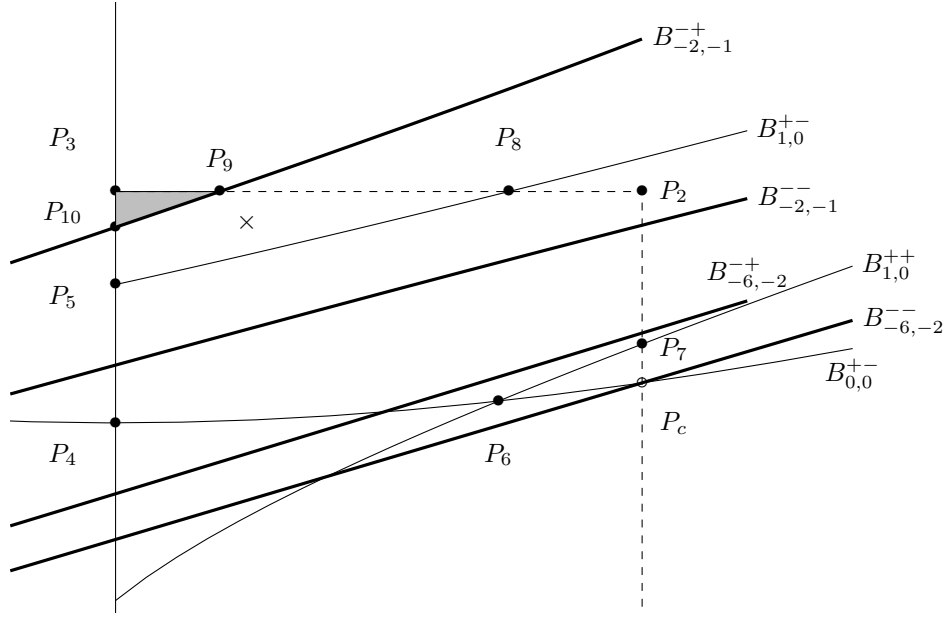

We need to prove
$$B_{-6, -2}^{-+} \geq B_{1, 0}^{++}\text{ for } a \in [x(P_6), 1/2],
\text{i.e.}, 20\,{a}^{2}-100\,a+57 \geq 0: \text{case }a_1;$$
$$B_{1, 0}^{++} \geq B_{-6, -2}^{--}\text{ for } a \in [x(P_6), 1/2],
\text{i.e.}, -605\,{a}^{2}+4488\,a-874 \geq 0; \text{case b}.$$

For $P_3$, $P_{5}$, $P_{8}$, there are no common pairs, but $(-2, -1)$
covers $\widehat{P_{5}P_{8}}$
that we add to Figure~\ref{real11-step3}. This leaves with two new
points
$$P_{9}=B_{-2, -1}^{-+} \cap \{y=1/2\}=\left(2-1/22\,\sqrt {1749}\approx 0.099043, 1/2\right),$$
$$P_{10}=B_{-2, -1}^{-+} \cap \{x=0\}=\left(0, 1-1/22\,\sqrt {138}\approx 0.466030\right).$$
We need to prove
$$B_{-2, -1}^{-+} \geq B_{1, 0}^{+-}\text{ for } a \in [0, 1/2], \text{i.e.}, 
-968\,{a}^{2}+3476\,a+1527 \geq 0: \text{ case b};$$
$$B_{1, 0}^{+-} \geq B_{-2, -1}^{--}\text{ for } a \in [0, 1/2], \text{i.e.}, 
4\,{a}^{2}-4\,a+9 \geq 0: \text{case }-.$$

\medskip
\noindent
{\em Step 4:} we are left with covering $\overline{P_{9}P_3P_{10}} \cup
\widehat{P_{10}P_{9}}$.

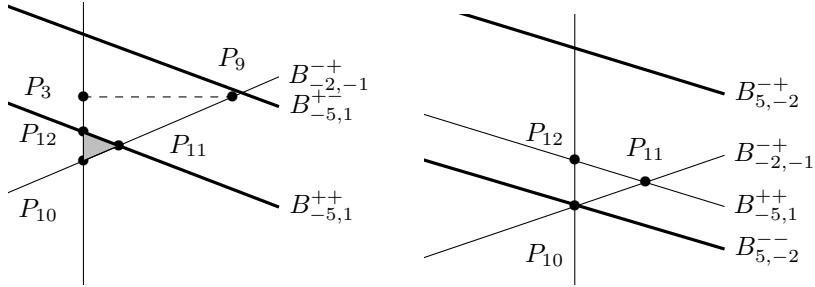
\begin{figure}[hbt]
\begin{center}
\begin{tikzpicture}[>=latex,xscale=20,yscale=25]
\clip (-0.05, 0.4) rectangle (0.2, 0.55);
\draw[->] (0, -1) -- (0, 1.5) node[very near end,left] {$b$};

\coordinate (P_{10}) at (0.000000, 0.466030);
\node at (P_{10}) {$\bullet$};
       \node at (-0.03, 0.44) {$P_{10}$};
\coordinate (P_{11}) at (0.023549, 0.474062);
\node at (P_{11}) {$\bullet$};
       \node at (0.07, 0.4741) {$P_{11}$};
\coordinate (P_{12}) at (0.000000, 0.481288);
\node at (P_{12}) {$\bullet$};
       \node at (-0.03, 0.48) {$P_{12}$};

\draw[fill=lightgray] (P_{12}) -- (P_{11}) -- (P_{10}) -- (P_{12});

\draw[domain=-.1:0.13]
plot (\x, {(0.00000) * \x + (1.00000) + (-1.00000) * sqrt((0.09091)*\x*\x+(-0.36364)*\x+(0.28512))})
 node[right] {$B_{-2, -1}^{- +}$};

\draw[domain=-.1:.13,very thick]
plot (\x, {(0.00000) * \x + (-1.00000) + (1.00000) * sqrt((0.09091)*\x*\x+(-0.90909)*\x+(2.19421))})
 node[right] {$B_{-5, 1}^{+ +}$};

\draw[domain=-.1:.13,very thick]
plot (\x, {(0.00000) * \x + (-1.00000) + (1.00000) * sqrt((0.09091)*\x*\x+(-0.90909)*\x+(2.35124))})
 node[right] {$B_{-5, 1}^{+ -}$};

\draw[domain=-.1:.13,very thick]
plot (\x, {(0.00000) * \x + (-1.00000) + (-1.00000) * sqrt((0.09091)*\x*\x+(-0.90909)*\x+(2.19421))})
 node[right] {$B_{-5, 1}^{- +}$};

\draw[domain=-.1:.13,very thick]
plot (\x, {(0.00000) * \x + (-1.00000) + (-1.00000) * sqrt((0.09091)*\x*\x+(-0.90909)*\x+(2.35124))})
 node[right] {$B_{-5, 1}^{- -}$};

\draw[dashed] (0, 0.5) -- (0.1, 0.5);
\node (P_3) at (0, 0.5) {$\bullet$}; \node at (-0.03, 0.505) {$P_3$};
\node (P_{9}) at (0.0990, 0.5000) {$\bullet$};
        \node at (0.0990, 0.52) {$P_{9}$};

\end{tikzpicture}
\hspace*{3mm}
\begin{tikzpicture}[>=latex,xscale=40,yscale=40]
\clip (-0.05, 0.44) rectangle (0.1, 0.53); 
\draw[->] (0, -1) -- (0, 1.5) node[very near end,left] {$b$};

\draw[domain=-.1:0.05] plot (\x, {1-sqrt(((\x+(-2))*(\x+(-2))-0.8636)/11)}) node[right] {$B_{-2, -1}^{-+}$};

\draw[domain=-.1:.05]
plot (\x, {(0.00000) * \x + (-1.00000) + (1.00000) * sqrt((0.09091)*\x*\x+(-0.90909)*\x+(2.19421))})
 node[right] {$B_{-5, 1}^{+ +}$};

\draw[domain=-.1:.05,very thick]
plot (\x, {(0.00000) * \x + (2.00000) + (1.00000) * sqrt((0.09091)*\x*\x+(0.90909)*\x+(2.19421))})
 node[right] {$B_{5, -2}^{+ +}$};

\draw[domain=-.1:.05,very thick]
plot (\x, {(0.00000) * \x + (2.00000) + (1.00000) * sqrt((0.09091)*\x*\x+(0.90909)*\x+(2.35124))})
 node[right] {$B_{5, -2}^{+ -}$};

\draw[domain=-.1:.05,very thick]
plot (\x, {(0.00000) * \x + (2.00000) + (-1.00000) * sqrt((0.09091)*\x*\x+(0.90909)*\x+(2.19421))})
 node[right] {$B_{5, -2}^{- +}$};

\draw[domain=-.1:.05,very thick]
plot (\x, {(0.00000) * \x + (2.00000) + (-1.00000) * sqrt((0.09091)*\x*\x+(0.90909)*\x+(2.35124))})
 node[right] {$B_{5, -2}^{- -}$};

\node (P_{10}) at (0.0000, 0.4660) {$\bullet$};
         \node at (-0.01, 0.45) {$P_{10}$};
\node (P_{11}) at (0.0235, 0.4741) {$\bullet$};
         \node at (0.0235, 0.485) {$P_{11}$};
\node (P_{12}) at (0.0000, 0.4813) {$\bullet$};
         \node at (-0.01, 0.49) {$P_{12}$};
\end{tikzpicture}
\end{center}
\caption{The case $m=11$: step 4 (left), step 5 (right). \label{real11-step5}}
\end{figure}
We can use $(-5, 1)$ and we find two new points
$$P_{11}=B_{-5, 1}^{++} \cap B_{-2, -1}^{-+} = \left(7/2-{\frac
{3\,\sqrt {1645}}{35}}\approx
0.0235,{\frac {9\,\sqrt {1645}}{770}}\approx 0.4741\right),$$
$$P_{12} = B_{-5, 1}^{++} \cap \{x=0\} = \left(0, -1+{\frac {3\,\sqrt
{118}}{22}}\approx 0.4813\right),$$
see Figure~\ref{real11-step5}. 
First, we prove that they cover $\widehat{P_{11} P_{9}}$ with
$$B_{-5, 1}^{+-} \geq B_{-2, -1}^{-+}\text{ for } a \in [x(P_{11}, x(P_{9})], \text{i.e.}, 
4235\,{a}^{2}-29018\,a+3009 \geq 0: \text{case }a_1;$$
$$B_{-2, -1}^{-+} \geq B_{-5, 1}^{++}\text{ for } a \in [x(P_{11}, x(P_{9})], \text{i.e.}, 
-140\,{a}^{2}+980\,a-23 \geq 0: \text{case b}.$$
For $P_3P_9$, we get
$$B_{-5, 1}^{+-} \geq 1/2\text{ for } a \in [0, x(P_9)], \text{i.e.}, 
44\,{a}^{2}-440\,a+49 \geq 0: \text{case }a_1;$$ 
$$B_{-5, 1}^{++} \leq 1/2\text{ for } a \in [0, x(P_9)], \text{i.e.}, 
44\,{a}^{2}-440\,a-27 \geq 0: \text{case b}.$$ 

\medskip
\noindent
{\em Step 5:} we are left with the region $\widehat{P_{10}P_{11}}
\overline{P_{11}P_{12} P_{10}}$.

We decide to cover $P_{11}$ and $P_{12}$ (again) with $(5, -2)$,
see Figure~\ref{real11-step5}. This
leaves a very tiny region not covered: $P_{10}P_{13}P_{14}$ (note $P_{10}$
is {\em not covered}) with
$$P_{13} = B_{5, -2}^{--} \cap \{x=0\}=\left(0, 2-1/22\,\sqrt {1138}\approx 0.4666\right),$$
$$P_{14} = B_{5, -2}^{--} \cap B_{-2, -1}^{-+}=\left(-{\frac{73}{44}}+1/44\,\sqrt {5335}\approx
0.0009,{\frac{67}{44}}-{\frac {7\,\sqrt {5335}}{484}}\approx 0.4663\right).$$
To prove the covering, we need
$$B_{5, -2}^{-+} \geq B_{-5, 1}^{++}\text{ for } a \in [0, x(P_{11})], \text{i.e.}, 
4\,{a}^{2}+243 \geq 0: \text{case }-;$$
$$B_{-5, 1}^{++} \geq B_{5, -2}^{--}\text{ for } a \in [0, x(P_{11})], \text{i.e.}, 
-242\,{a}^{2}-4180\,a+5809 \geq 0: \text{case b}.$$

\medskip
\noindent
{\em Step 6:} We need cover the region
$\widehat{P_{10}P_{14}P_{13}}\overline{P_{13}P_{10}}$, see
Figure~\ref{real11-step6}.
\begin{figure}[hbt]
\begin{center}
\iftrue
\begin{tikzpicture}[>=latex,xscale=300,yscale=300]
\clip (-0.005, 0.455) rectangle (0.01, 0.47);
\else
\begin{tikzpicture}[>=latex,xscale=200,yscale=200]
\clip (-0.01, 0.45) rectangle (0.015, 0.48);
\fi
\draw[->] (0, 0.45) -- (0, 0.5);

\draw[domain=-0.1:0.005] plot (\x, {1-sqrt(((\x+(-2))*(\x+(-2))-0.8636)/11)}) node[right] {$B_{-2, -1}^{-+}$};


\draw[domain=-0.1:0.005] plot (\x, {2-sqrt(((\x+(5))*(\x+(5))+0.8636)/11)}) node[below] {$B_{5, -2}^{--}$};

\node (P_{13}) at (0.0000, 0.4666) {$\bullet$};
         \node at (-0.0025, 0.4666) {$P_{13}$};
\node (P_{14}) at (0.0009, 0.4663) {$\bullet$}; \node at (0.001, 0.465) {$P_{14}$};

\node (P_{10}) at (0.0000, 0.4660) {$\bullet$};
         \node at (-0.0015, 0.4645) {$P_{10}$};

\draw[domain=-.1:.005,very thick]
plot (\x, {(0.00000) * \x + (8.00000) + (-1.00000) * sqrt((0.09091)*\x*\x+(4.54545)*\x+(56.73967))})
 node[right] {$B_{25, -8}^{- +}$};

\draw[domain=-.1:.005,very thick]
plot (\x, {(0.00000) * \x + (8.00000) + (-1.00000) * sqrt((0.09091)*\x*\x+(4.54545)*\x+(56.89669))})
 node[right] {$B_{25, -8}^{- -}$};

\end{tikzpicture}
\end{center}
\caption{The case $m=11$: step 6. \label{real11-step6}}
\end{figure}
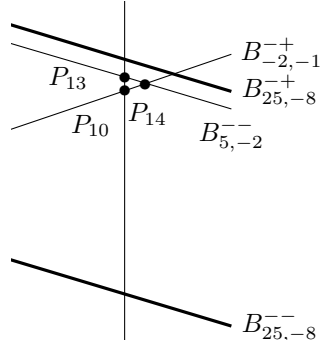
We may use $(25, -8)$, as found by computations and need to prove
$$B_{25, -8}^{-+} \geq B_{5, -2}^{--}\text{ for } a \in [0, x(P_{14})], \text{i.e.}, 
-1936\,{a}^{2}-41360\,a+6503 \geq 0: \text{ case b};$$
$$B_{5, -2}^{-+} \geq B_{25, -8}^{--}\text{ for } a \in [0, x(P_{14})], \text{i.e.}, 
1936\,{a}^{2}+74800\,a+495097 \geq 0: \text{case }a_2.$$
This ends the proof of the theorem. $\Box$

\subsection{The case $m=19$}

Berg~\cite{Berg1935} gave the minimum by giving a list of inequalities
(without proof); see also \cite{Varnavides1952}. We may anticipate
a harder case.

\begin{theorem}
For $m=19$, the square $\mathcal{S}_0$ can be covered by the pairs
$$(0, 0), (1, 0), (-2, 0), (2, -1), (-7, 1), (-3, -1), (7, -2), (-6,
1), (991, 227),$$
$$(-19, 4), (-80, 18), (-430, -99), (90, -21).$$
\end{theorem}

There is one exceptional point $P_c = (0, 20/57)$, which is weakly
covered:
$$C(P_c, 10^3) = \{[-991, 227], [-29, -7], [-3, -1], [3, -1], [29,
-7], [991, 227]\}.$$
The first points of interest are:
$$P_4 = \left(0, {\frac {\sqrt {170}}{57}} \approx 0.2287\right),
P_5 = \left(1/2, {\frac {\sqrt {851}}{114}} \approx 0.2559\right)$$
where $P_4$ and $P_5$ are intersection points with $B_{0, 0}^{+, -}$,
see Figure~\ref{real19-step2}.
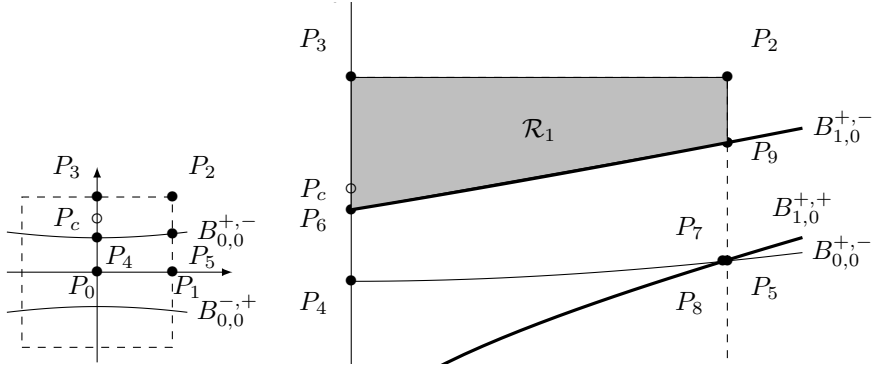
\begin{figure}[hbt]
\begin{tikzpicture}[>=latex,xscale=2,yscale=2]
\draw[->] (-0.6, 0) -- (0.9, 0); 
\draw[->] (0, -0.6) -- (0, 0.7); 
\draw[domain=-0.6:0.6] plot (\x, {sqrt((\x*\x+0.994)/19)}) node[right] {$B_{0, 0}^{+, -}$};
\draw[domain=-0.6:0.6] plot (\x, {-sqrt((\x*\x+0.994)/19)}) node[right] {$B_{0, 0}^{-, +}$};
\draw[dashed] (-0.5, -0.5) rectangle (0.5, 0.5);
\node (P0) at (0, 0) {$\bullet$};
     \node at (-0.1, -0.1) {$P_{0}$};
\node (P1) at (1/2, 0) {$\bullet$};
     \node at (0.6, -0.1) {$P_{1}$};
\node (P_2) at (0.5, 0.5) {$\bullet$}; \node at (0.7, 0.7) {$P_2$};
\node (P_3) at (0, 0.5) {$\bullet$}; \node at (-0.2, 0.7) {$P_3$};
\node (P_4) at (0, 0.229) {$\bullet$}; \node at (0.15, 0.1) {$P_4$};
\node at (0.5, 0.256) {$\bullet$}; \node at (0.7, 0.1) {$P_5$};
\node (P_c) at (0.0000, 0.3509) {$\circ$}; \node at (-0.2, 0.3509) {$P_c$};
\end{tikzpicture}
\begin{tikzpicture}[>=latex,xscale=10,yscale=10]
\clip (-0.1, 0.12) rectangle (0.7, 0.6);
\draw[->] (0, 0) -- (0.7, 0) node[very near end,above] {$a$};
\draw[->] (0, 0) -- (0, 0.7) node[very near end,left] {$b$};
\draw[domain=0:0.6] plot (\x, {sqrt((\x*\x+0.994)/19)}) node[right]
{$B_{0, 0}^{+, -}$};
\node (P_4) at (0, 0.229) {$\bullet$}; \node at (-0.05, 0.2) {$P_4$};
\coordinate (P2) at (0.500000, 0.500000);
\coordinate (P3) at (0.000000, 0.500000);
\node at (0.5, 0.256) {$\bullet$}; \node at (0.55, 0.22) {$P_5$};
\coordinate (P_{6}) at (0.000000, 0.323968); \node at (P_{6}) {$\bullet$};
\coordinate (P_{7}) at (0.500000, 0.413213); \node at (P_{7}) {$\bullet$};
\coordinate (P_{8}) at (0.494152, 0.255295); \node at (P_{8}) {$\bullet$};
\coordinate (P_{9}) at (0.500000, 0.413213); \node at (P_{9}) {$\bullet$};

\draw[fill=lightgray] (P2) -- (P3) -- (P_{6}) -- (P_{9}) -- (P2);

\node (P_c) at (0.0000, 0.3509) {$\circ$}; \node at (-0.05, 0.3509) {$P_c$};
 \node at (P2) {$\bullet$}; \node at (0.55, 0.55) {$P_2$};
 \node at (P3) {$\bullet$}; \node at (-0.05, 0.55) {$P_3$};
 \node at (0.25, 0.43) {$\mathcal{R}_1$};

\draw[domain=0:0.6,very thick] plot (\x,
{0+sqrt(((\x+(1))*(\x+(1))+0.9942)/19)}) node[right] {$B_{1, 0}^{+, -}$};
\draw[domain=0:0.6,very thick] plot (\x,
{0+sqrt(((\x+(1))*(\x+(1))-0.9942)/19)}) node[above] {$B_{1, 0}^{+, +}$};
    \node at (-0.05, 0.310) {$P_{6}$};
     \node at (0.45, 0.3) {$P_{7}$};
     \node at (0.45, 0.2) {$P_{8}$};
     \node at (0.55, 0.4) {$P_{9}$};

\draw[dashed] ( 0.5, 0) -- ( 0.5, 0.5);
\draw[dashed] (0, 0.5) -- (0.5, 0.5);
\end{tikzpicture}
\caption{The case $m=19$: step 1 (left), step 2 (right). \label{real19-step2}}
\end{figure}

\noindent
{\em Step 2:}
We decide to use $(1, 0)$ to cover some more part of $\mathcal{S}_0$
including $P_4$, see Figure~\ref{real19-step2}.
This creates new points
$$P_6 = B_{1, 0}^{+-} \cap \{x=0\}=\left(0,{\frac {\sqrt {341}}{57}}\approx 0.3240\right),$$
$$P_7 = B_{1, 0}^{++} \cap \{x=1/2\}=\left(1/2,{\frac {\sqrt {859}}{114}}\approx 0.2571\right),$$
$$P_8 = B_{1, 0}^{++} \cap B_{0, 0}^{+-}=\left({\frac{169}{342}}\approx 0.4942,{\frac {\sqrt
{2751979}}{6498}}\approx 0.2553\right),$$
$$P_9 = B_{1, 0}^{+-} \cap \{x=1/2\}=\left(1/2, {\frac {\sqrt {2219}}{114}}\approx 0.4132\right)$$
We need to prove that
$$B_{1, 0}^{+, -} \geq B_{0, 0}^{+, -} \geq B_{1, 0}^{+, +}$$
for $a \in [0, x(P_8)]$, that is
$$2 a+1 \geq 0, 169-342 a \geq 0.$$
We need to cover the region $\mathcal{R}_1 = P_2P_3P_6P_9$ and the very
tiny region $P_5P_7P_8$. The latter
is coverable by $(5, 1)$, see Figure~\ref{real19-step3}. We have to prove
$$B_{5, 1}^{+-} \geq B_{1, 0}^{++}\text{ for } a \in [0, 1/2], \text{i.e.}, 
-350892\,{a}^{2}-1175112\,a+1415029 \geq 0: \text{case b};$$ 
$$B_{5, 1}^{+-} \geq B_{0, 0}^{+-}\text{ for } a \in [1/3, 1/2], \text{i.e.}, 
54\,{a}^{2}+270\,a-89 \geq 0: \text{case }a_2;$$ 
$$B_{0, 0}^{+-} \geq B_{5, 1}^{++}\text{ for } a \in [1/3, 1/2], \text{i.e.}, 
-175446\,{a}^{2}-586530\,a+434681 \geq 0: \text{case b}.$$ 
(The value $1/3$ is arbitrary and is enough for our proof.)

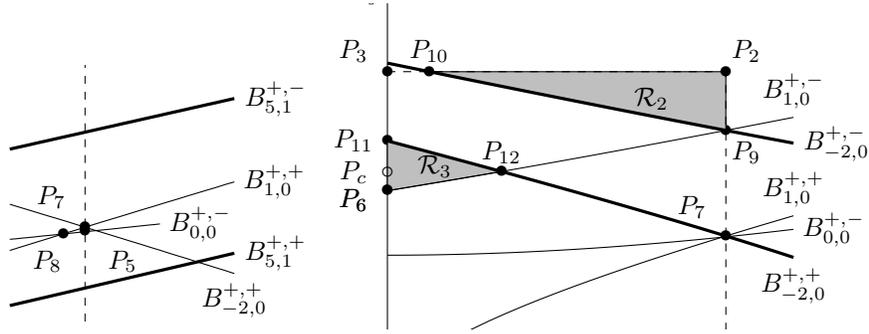
\begin{figure}[hbt]
\begin{tikzpicture}[>=latex,xscale=50,yscale=50]
\clip (0.475, 0.23) rectangle (0.56, 0.3);
\draw[->] (0, 0) -- (0.7, 0) node[very near end,above] {$a$};
\draw[->] (0, 0) -- (0, 0.7) node[very near end,left] {$b$};
\node (P_c) at (0.0000, 0.3509) {$\circ$}; \node at (-0.05, 0.3509) {$P_c$};
\node at (0.5, 0.256) {$\bullet$}; \node at (0.51, 0.248) {$P_5$};
\draw[domain=0.48:0.52] plot (\x, {sqrt((\x*\x+0.994)/19)}) node[right]
{$B_{0, 0}^{+, -}$};

\node (P_{7}) at (0.5000, 0.2571) {$\bullet$}; \node at (0.491, 0.265) {$P_{7}$};
\draw[domain=0.48:0.54] plot (\x,
{0+sqrt(((\x+(-2))*(\x+(-2))-0.9942)/19)}) node[below] {$B_{-2, 0}^{+,
+}$};

\draw[domain=0.48:0.54] plot (\x, {0+sqrt(((\x+(1))*(\x+(1))-0.9942)/19)}) node[right] {$B_{1, 0}^{+,+}$};
\node (P_{8}) at (0.4942, 0.2553) {$\bullet$}; \node at (0.49, 0.248) {$P_{8}$};

\draw[domain=0.48:0.54,very thick] plot (\x,
{-1+sqrt(((\x+(5))*(\x+(5))+0.9942)/19)}) node[right] {$B_{5, 1}^{+, -}$};
\draw[domain=0.48:0.54,very thick] plot (\x,
{-1+sqrt(((\x+(5))*(\x+(5))-0.9942)/19)}) node[right] {$B_{5, 1}^{+, +}$};

\draw[dashed] ( 0.5, 0) -- ( 0.5, 0.5);
\draw[dashed] (0, 0.5) -- (0.5, 0.5);
\end{tikzpicture}
\begin{tikzpicture}[>=latex,xscale=9,yscale=9]
\clip (-0.1, 0.12) rectangle (0.8, 0.6);
\draw[->] (0, 0) -- (0.7, 0) node[very near end,above] {$a$};
\draw[->] (0, 0) -- (0, 0.7) node[very near end,left] {$b$};
\draw[domain=0:0.6] plot (\x, {sqrt((\x*\x+0.994)/19)}) node[right]
{$B_{0, 0}^{+, -}$};
\coordinate (P2) at (0.500000, 0.500000);
\node (P_3) at (0, 0.5) {$\bullet$};
      \node at (-0.05, 0.53) {$P_3$};

\coordinate (P_{6}) at (0.000000, 0.323968); \node at (P_{6}) {$\bullet$};
    \node at (-0.05, 0.310) {$P_{6}$};
\coordinate (P_{9}) at (0.500000, 0.413213); \node at (P_{9}) {$\bullet$};
     \node at (0.53, 0.38) {$P_{9}$};

\coordinate (P_{10}) at (0.061999, 0.500000); \node at (P_{10}) {$\bullet$};
\coordinate (P_{11}) at (0.000000, 0.397747); \node at (P_{11}) {$\bullet$};
\coordinate (P_{12}) at (0.168616, 0.352421); \node at (P_{12}) {$\bullet$};
 \node at (0.0620, 0.53) {$P_{10}$};
 \node at (-0.05, 0.3977) {$P_{11}$};
 \node at (0.1686, 0.38) {$P_{12}$};

\draw[fill=lightgray] (P_{9}) -- (P2) -- (P_{10}) -- (P_{9});
\draw[fill=lightgray] (P_{6}) -- (P_{11}) -- (P_{12}) -- (P_{6});

\node (P_c) at (0.0000, 0.3509) {$\circ$}; \node at (-0.05, 0.3509) {$P_c$};
 \node at (P2) {$\bullet$}; \node at (0.53, 0.53) {$P_2$};
 \node at (0.39, 0.46) {$\mathcal{R}_2$};
 \node at (0.07, 0.36) {$\mathcal{R}_3$};

\draw[domain=0:0.6] plot (\x,
{0+sqrt(((\x+(1))*(\x+(1))+0.9942)/19)}) node[above] {$B_{1, 0}^{+, -}$};
\draw[domain=0:0.6] plot (\x,
{0+sqrt(((\x+(1))*(\x+(1))-0.9942)/19)}) node[above] {$B_{1, 0}^{+, +}$};
\node (P_{6}) at (0.0000, 0.3240) {$\bullet$}; \node at (-0.05, 0.310) {$P_{6}$};
\node (P_{7}) at (0.5000, 0.2571) {$\bullet$}; \node at (0.45, 0.3) {$P_{7}$};

\draw[domain=0:0.6,very thick] plot (\x,
{0+sqrt(((\x+(-2))*(\x+(-2))+0.9942)/19)}) node[right] {$B_{-2, 0}^{+,
-}$};
\draw[domain=0:0.6,very thick] plot (\x,
{0+sqrt(((\x+(-2))*(\x+(-2))-0.9942)/19)}) node[below] {$B_{-2, 0}^{+,
+}$};

\draw[dashed] ( 0.5, 0) -- ( 0.5, 0.5);
\draw[dashed] (0, 0.5) -- (0.5, 0.5);
\end{tikzpicture}
\caption{The case $m=19$: end of step 2 (left), step 3 (right). \label{real19-step3}}
\end{figure}

\medskip
\noindent
{\em Step 3:} $P_3, P_7$ and $P_9$ can be covered by $(-2, 0)$, see
Figure~\ref{real19-step3}.
This creates new points
$$P_{10} = B_{-2, 0}^{+-} \cap \{y=1/2\}=\left(2-{\frac {\sqrt {48811}}{114}}\approx 0.0620, 1/2\right),$$
$$P_{11} = B_{-2, 0}^{++} \cap \{x=0\}=\left(0, {\frac {\sqrt {514}}{57}} \approx 0.3977\right),$$
$$P_{12} = B_{-2, 0}^{++} \cap B_{1, 0}^{+-}=\left({\frac{173}{1026}}\approx 0.1686,{\frac {\sqrt
{47198299}}{19494}}\approx 0.3524\right).$$
Proofs are:
$$B_{-2, 0}^{+-} \geq B_{1, 0}^{+-}\text{ for } a \in [173/1026, 1/2], \text{i.e.}, 
-2\,a+1 \geq 0;$$
$$B_{1, 0}^{+-} \geq B_{-2, 0}^{++}\text{ for } a \in [173/1026, 1/2], \text{i.e.}, 1026\,a-173 \geq 0.$$
We also need to prove the covering of $P_3P_{10}$:
$$B_{-2, 0}^{+-} \geq 1/2\text{ for } a \in [0, x(P_{10})], \text{i.e.}, 
684\,{a}^{2}-2736\,a+167 \geq 0: \text{case }a_1;$$ 
$$B_{-2, 0}^{++} \leq 1/2\text{ for } a \in [0, x(P_{10})], \text{i.e.}, 
684\,{a}^{2}-2736\,a-1193 \geq 0: \text{case b}.$$ 

We still have to cover $\mathcal{R}_2 =
\overline{P_9P_2P_{10}}\widehat{P_{10}P_9}$ and $\mathcal{R}_3 =
\overline{P_{11}P_cP_6}\widehat{P_6 P_{12}}$.

\subsubsection{Covering $\mathcal{R}_2$}

\noindent
{\em Step 4:} $(2, -1)$ covers $P_9$ and $P_{10}$, see
Figure~\ref{real19-step4}.

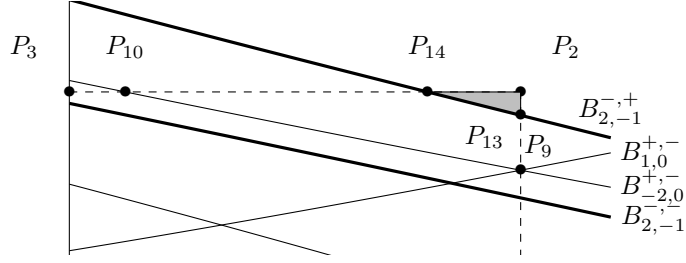
\begin{figure}[hbt]
\begin{tikzpicture}[>=latex,xscale=12,yscale=12]
\clip (-0.1, 0.32) rectangle (0.7, 0.6);
\draw[->] (0, 0) -- (0.7, 0); 
\draw[->] (0, 0) -- (0, 0.7); 
\draw[domain=0:0.6] plot (\x, {sqrt((\x*\x+0.994)/19)}) node[right]
{$B_{0, 0}^{+, -}$};
\node (P_4) at (0, 0.229) {$\bullet$}; \node at (-0.05, 0.2) {$P_4$};
\coordinate (P2) at (0.500000, 0.500000);
\node at (P2) {$\bullet$}; \node at (0.55, 0.55) {$P_2$};
\node (P_3) at (0, 0.5) {$\bullet$}; \node at (-0.05, 0.55) {$P_3$};

\coordinate (P_{13}) at (0.500000, 0.474050);
\coordinate (P_{14}) at (0.396696, 0.500000);

\draw[fill=lightgray] (P_{13}) -- (P2) -- (P_{14}) -- (P_{13});

\node at (P_{13}) {$\bullet$}; \node at (0.4600, 0.45) {$P_{13}$};
\node at (P_{14}) {$\bullet$}; \node at (0.3967, 0.55000) {$P_{14}$};

\draw[domain=0:0.6] plot (\x, {0+sqrt(((\x+(1))*(\x+(1))+0.9942)/19)})
node[right] {$B_{1, 0}^{+, -}$};
\node (P_{8}) at (0.4942, 0.2553) {$\bullet$}; \node at (0.45, 0.2) {$P_{8}$};
\node (P_{9}) at (0.5000, 0.4132) {$\bullet$}; \node at (0.52, 0.44) {$P_{9}$};

\draw[domain=0:0.6] plot (\x,
{0+sqrt(((\x+(-2))*(\x+(-2))+0.9942)/19)}) node[right] {$B_{-2, 0}^{+,
-}$};
\draw[domain=0:0.6] plot (\x,
{0+sqrt(((\x+(-2))*(\x+(-2))-0.9942)/19)}) node[right] {$B_{-2, 0}^{+,
+}$};
\node (P_{10}) at (0.0620, 0.5000) {$\bullet$}; \node at (0.0620, 0.5500) {$P_{10}$};

\draw[domain=0:0.6,very thick] plot (\x,
{1-sqrt(((\x+(2))*(\x+(2))+0.9942)/19)}) node[right] {$B_{2, -1}^{-, -}$};
\draw[domain=0:0.6,very thick] plot (\x,
{1-sqrt(((\x+(2))*(\x+(2))-0.9942)/19)}) node[above] {$B_{2, -1}^{-, +}$};

\draw[dashed] ( 0.5, 0) -- ( 0.5, 0.5);
\draw[dashed] (0, 0.5) -- (0.5, 0.5);
\end{tikzpicture}
\caption{The case $m=19$: step 4. \label{real19-step4}}
\end{figure}
We have to prove
$$B_{2, -1}^{-+} \geq B_{-2, 0}^{+-}\text{ for } a \in [0, 1/2], \text{i.e.}, 
-350892\,{a}^{2}-930240\,a+1782337 \geq 0: \text{case b};$$ 
$$B_{-2, 0}^{+-} \geq B_{2, -1}^{--}\text{ for } a \in [0, 1/2], \text{i.e.}, 
108\,{a}^{2}+167 \geq 0: \text{case -}.$$ 
For the covering of $P_{10}P_{14}$, we get
$$B_{2, -1}^{-+} \geq 1/2\text{ for } a \in [x(P_{10}), x(P_{14})], \text{i.e.}, 
-684\,{a}^{2}-2736\,a+1193 \geq 0: \text{case b};$$ 
$$B_{2, -1}^{--} \leq 1/2\text{ for } a \in [x(P_{10}), x(P_{14})], \text{i.e.}, 
-684\,{a}^{2}-2736\,a-167 \geq 0: \text{case }a_2.$$ 
We find two new points
$$P_{13} = B_{2, -1}^{-+} \cap \{x=1/2\}= \left(1/2,1-{\frac {\sqrt {3595}}{114}}\approx 0.4740\right),$$
$$P_{14} = B_{2, -1}^{-+} \cap \{y=1/2\}=\left(-2+{\frac {\sqrt {74651}}{114}}\approx 0.3967,1/2\right).$$

\medskip
\noindent
{\em Step 5:} we have to cover $P_{13}P_2P_{14}$, which can be done
using $(-7, 1)$, see Figure~\ref{real19-step6}.
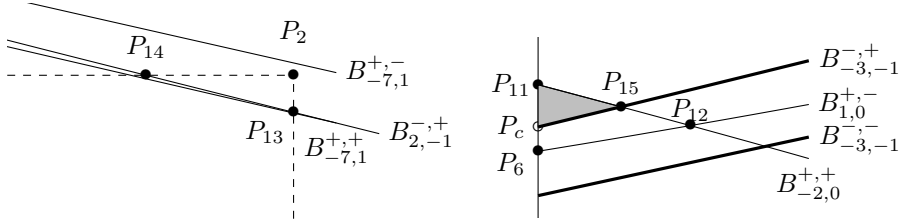
\begin{figure}[hbt]
\begin{tikzpicture}[>=latex,xscale=19,yscale=19]
\clip (0.3, 0.4) rectangle (0.63, 0.55);
\draw[->] (0, 0) -- (0.7, 0) node[very near end,above] {$a$};
\draw[->] (0, 0) -- (0, 0.7) node[very near end,left] {$b$};

\node (P_2) at (0.5, 0.5) {$\bullet$}; \node at (0.5, 0.53) {$P_2$};

\draw[domain=0:0.56] plot (\x, {1-sqrt(((\x+(2))*(\x+(2))-0.9942)/19)})
node[right] {$B_{2, -1}^{-, +}$};
\node (P_{13}) at (0.5000, 0.4740) {$\bullet$};
         \node at (0.4800, 0.46) {$P_{13}$};
\node (P_{14}) at (0.3967, 0.5000) {$\bullet$};
         \node at (0.3967, 0.52000) {$P_{14}$};

\draw[domain=0:0.53] plot (\x,
{-1+sqrt(((\x+(-7))*(\x+(-7))+0.9942)/19)}) node[right] {$B_{-7, 1}^{+, -}$};
\draw[domain=0:0.53] plot (\x,
{-1+sqrt(((\x+(-7))*(\x+(-7))-0.9942)/19)}) node[below] {$B_{-7, 1}^{+, +}$};

\draw[dashed] ( 0.5, 0) -- ( 0.5, 0.5);
\draw[dashed] (0, 0.5) -- (0.5, 0.5);
\end{tikzpicture}
\begin{tikzpicture}[>=latex,xscale=12,yscale=12]
\clip (-0.055, 0.25) rectangle (0.45, 0.45);
\draw[->] (0, 0) -- (0, 0.7) node[very near end,left] {$b$};

\coordinate (P_c) at (0.0000, 0.3509); \node at (P_c) {$\circ$};
\coordinate (P_{11}) at (0.000000, 0.397747); \node at (P_{11}) {$\bullet$};
\coordinate (P_{15}) at (0.091809, 0.373253); \node at (P_{15}) {$\bullet$};

\draw[fill=lightgray] (P_{11}) -- (P_c) -- (P_{15}) -- (P_{11});

      \node at (-0.03, 0.3509) {$P_c$};

\draw[domain=0:0.3] plot (\x, {0+sqrt(((\x+(1))*(\x+(1))+0.9942)/19)})
node[right] {$B_{1, 0}^{+, -}$};
\node (P_{6}) at (0.0000, 0.3240) {$\bullet$};
        \node at (-0.03, 0.310) {$P_{6}$};
\draw[domain=0:0.3] plot (\x,
{0+sqrt(((\x+(-2))*(\x+(-2))-0.9942)/19)}) node[below] {$B_{-2, 0}^{+, +}$};

    \node at (-0.03, 0.3977) {$P_{11}$};
\node (P_{12}) at (0.1686, 0.3524) {$\bullet$};
         \node at (0.1686, 0.37) {$P_{12}$};

\draw[domain=0:0.3,very thick] plot (\x,
{1-sqrt(((\x+(-3))*(\x+(-3))+0.9942)/19)}) node[right] {$B_{-3,
-1}^{-, -}$};
\draw[domain=0:0.3,very thick] plot (\x,
{1-sqrt(((\x+(-3))*(\x+(-3))-0.9942)/19)}) node[right] {$B_{-3,
-1}^{-, +}$};
         \node at (0.0918, 0.4) {$P_{15}$};
\end{tikzpicture}
\caption{The case $m=19$: step 5 (left), step 6 (right). \label{real19-step6}}
\end{figure}
We need to prove that $P_2P_{14}$ is covered:
$$B_{-7, 1}^{+-} \geq 1/2\text{ for } a \in [x(P_{14}), 1/2], \text{i.e.}, 
684\,{a}^{2}-9576\,a+4955 \geq 0: \text{case }a_1;$$ 
$$B_{-7, 1}^{++} \leq 1/2\text{ for } a \in [x(P_{14}]), 1/2], \text{i.e.}, 
684\,{a}^{2}-9576\,a+3595 \geq 0: \text{case b}.$$ 
Also:
$$B_{-7, 1}^{+-} \geq B_{2, -1}^{-+}\text{ for } a \in [0, 1/2], \text{i.e.}, 
-116964\,{a}^{2}+1003428\,a+421651 \geq 0: \text{case b};$$ 
$$B_{2, -1}^{-+} \geq B_{-7, 1}^{++}\text{ for } a \in [0, 1/2], \text{i.e.}, 
36\,{a}^{2}-180\,a+85 \geq 0: \text{case }a_1.$$ 

\subsubsection{Covering $\mathcal{R}_3$}

\noindent
{\em Step 6:} We select $(-3, -1)$ and obtain Figure~\ref{real19-step6}.
And this creates
$$P_{15} = B_{-3, -1}^{-+} \cap B_{-2, 0}^{++}= \left(5/2-1/18\,\sqrt
{1879}\approx 0.0918,1/2-{\frac
{\sqrt {1879}}{342} \approx 0.3733}\right).$$
We need to prove
$$B_{-3, -1}^{-+} \geq B_{1, 0}^{+-}\text{ for } a \in [0, 1/2], \text{i.e.}, 
-350892\,{a}^{2}+1632024\,a+501205 \geq 0: \text{case b};$$ 
$$B_{1, 0}^{+-} \geq B_{-3, -1}^{--}\text{ for } a \in [0, 1/2], \text{i.e.}, 
108\,{a}^{2}-216\,a+275 \geq 0: \text{case -}.$$ 
This leaves us with $\widehat{P_{11}P_cP_{15}}$.

\medskip
\noindent
{\em Step 7:} in order to cover $P_{15}$,
we use $(7, -2)$ to cover $\widehat{P_{11}P_{15}}$, see
Figure~\ref{real19-step8}. This creates
$$P_{16} = B_{7, -2}^{--} \cap \{x=0\}=\left(0,2-{\frac {\sqrt {8549}}{57}} \approx 0.3779\right),$$
$$P_{17} = B_{7, -2}^{--} \cap B_{-3, -1}^{-+}=\left(
-{\frac{29402}{13851}}+{\frac{\sqrt{3647350219}}{27702}}\approx 0.0574,\right.$$
$$\left.{\frac{41893}{27702}}-{\frac {5\,\sqrt{3647350219}}{263169}}\approx 0.3648\right).$$
We need to prove
$$B_{7, -2}^{-+} \geq B_{-2, 0}^{++}\text{ for } a \in [0, 2/5], \text{i.e.}, 
36\,{a}^{2}+180\,a+85 \geq 0: \text{case }a_2;$$ 
$$B_{-2, 0}^{++} \geq B_{7, -2}^{--}\text{ for } a \in [0, 2/5], \text{i.e.}, 
-116964\,{a}^{2}-1003428\,a+421651 \geq 0: \text{case b}.$$ 

\begin{figure}[hbt]
\begin{tikzpicture}[>=latex,xscale=20,yscale=20]
\clip (-0.04, 0.3) rectangle (0.27, 0.45);
\draw[->] (0, 0) -- (0, 0.7) node[very near end,left] {$b$};

\coordinate (P_c) at (0.0000, 0.3509); \node at (P_c) {$\circ$};
\coordinate (P_{16}) at (0.000000, 0.377881); \node at (P_{16}) {$\bullet$};
\coordinate (P_{17}) at (0.057371, 0.364849); \node at (P_{17}) {$\bullet$};

\draw[fill=lightgray] (P_{17}) -- (P_c) -- (P_{16}) -- (P_{17});

      \node at (-0.02, 0.3509) {$P_c$};

\draw[domain=0:0.2] plot (\x,
{0+sqrt(((\x+(-2))*(\x+(-2))-0.9942)/19)}) node[right] {$B_{-2, 0}^{+, +}$};
\node (P_{11}) at (0.0000, 0.3977) {$\bullet$};
         \node at (-0.02, 0.3977) {$P_{11}$};

\draw[domain=0:0.2] plot (\x,
{1-sqrt(((\x+(-3))*(\x+(-3))-0.9942)/19)}) node[right] {$B_{-3,
-1}^{-, +}$};
\node (P_{15}) at (0.0918, 0.3733) {$\bullet$}; \node at (0.0918, 0.4) {$P_{15}$};

\draw[domain=0:0.2,very thick] plot (\x, {2+sqrt(((\x+(7))*(\x+(7))+0.9942)/19)}) node[right] {$B_{7, -2}^{+, -}$};
\draw[domain=0:0.2,very thick] plot (\x, {2-sqrt(((\x+(7))*(\x+(7))+0.9942)/19)}) node[below] {$B_{7, -2}^{-, -}$};
\draw[domain=0:0.2,very thick] plot (\x, {2-sqrt(((\x+(7))*(\x+(7))-0.9942)/19)}) node[right] {$B_{7, -2}^{-, +}$};
\node at (-0.02, 0.3779) {$P_{16}$};
\node at (0.0574, 0.35) {$P_{17}$};
\end{tikzpicture}
\begin{tikzpicture}[>=latex,xscale=60,yscale=60]
\clip (-0.015, 0.335) rectangle (0.09, 0.4);
\draw[->] (0, 0.32) -- (0, 0.4); 

\coordinate (P_c) at (0.0000, 0.3509); \node at (P_c) {$\circ$};
\coordinate (P_{18}) at (0.000000, 0.357355); \node at (P_{18}) {$\bullet$};
\coordinate (P_{19}) at (0.013610, 0.354189); \node at (P_{19}) {$\bullet$};

\draw[fill=lightgray] (P_{18}) -- (P_c) -- (P_{19}) -- (P_{18});

      \node at (-0.0075, 0.3509) {$P_c$};

\draw[domain=0:0.07] plot (\x,
{1-sqrt(((\x+(-3))*(\x+(-3))-0.9942)/19)}) node[right] {$B_{-3,
-1}^{-, +}$};


\draw[domain=0:0.07] plot (\x, {2-sqrt(((\x+(7))*(\x+(7))+0.9942)/19)}) node[below] {$B_{7, -2}^{-, -}$};
\node (P_{16}) at (0.0000, 0.3779) {$\bullet$};
         \node at (-0.0075, 0.3779) {$P_{16}$};
\node (P_{17}) at (0.0574, 0.3648) {$\bullet$}; \node at (0.0574, 0.36) {$P_{17}$};

\draw[domain=0:0.07,very thick] plot (\x, {-1+sqrt(((\x+(-6))*(\x+(-6))+0.9942)/19)}) node[right] {$B_{-6, 1}^{+, -}$};
\draw[domain=0:0.07,very thick] plot (\x, {-1+sqrt(((\x+(-6))*(\x+(-6))-0.9942)/19)}) node[right] {$B_{-6, 1}^{+, +}$};
\draw[domain=0:0.07] plot (\x,
{-1-sqrt(((\x+(-6))*(\x+(-6))-0.9942)/19)}) node[right] {$B_{-6, 1}^{-, +}$};
\node at (-0.0075, 0.359) {$P_{18}$};
\node at (0.0136, 0.347) {$P_{19}$};

\end{tikzpicture}
\caption{The case $m=19$: step 7 (left), step 8 (right). \label{real19-step8}}
\end{figure}
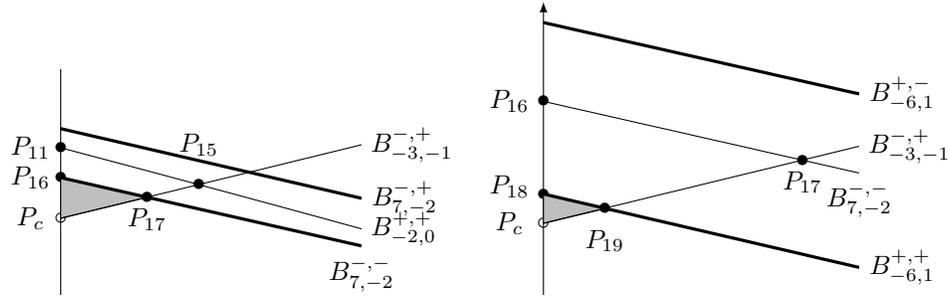

\medskip
\noindent
{\em Step 8:} we still need to cover $\widehat{P_cP_{17}P_{16}}$. Due
to the weak covering of $P_c$, we have to start a tedious process to
perform the covering of the region. At each step, some smaller region
will remain to be covered. This process has begun in Step 7 already.
We choose $(-6, 1)$ to cover at least $P_{16}$ and $P_{17}$, see
Figure~\ref{real19-step8}, which creates two points
$$P_{18}=B_{-6, 1}^{++} \cap \{x=0\}=\left(0,-1+{\frac {\sqrt {5986}}{57}} \approx 0.3574\right),$$
$$P_{19}=B_{-6, 1}^{++} \cap B_{-3, -1}^{-+}=\left(9/2-{\frac {\sqrt {813179}}{201}}\approx 0.0136,{\frac {\sqrt
{813179}}{2546}}\approx 0.3542\right).$$
We have to prove
$$B_{-6, 1}^{+-} \geq B_{7, -2}^{--}\text{ for } a \in [0, 2/5], \text{i.e.}, 
2\,{a}^{2}+2\,a+85 \geq 0: \text{case -};$$ 
$$B_{7, -2}^{--} \geq B_{-6, 1}^{++}\text{ for } a \in [0, 2/5], \text{i.e.}, 
-58482\,{a}^{2}+697338\,a+2892295 \geq 0: \text{case b}.$$ 

\medskip
\noindent
{\em Step 9:}
We compute
\begin{verbatim}
C[18]={[-321, -74], [-19, 4], [-6, 1], [6, 1], [19, 4],
   [321, -74]}
C[19]={[-430, -99], [-19, 4], [-6, 1], [-3, -1], [321, -74],
   [991, 227]}
\end{verbatim}
which shows that $P_c$ and $P_{19}$ may be covered by $(991, 227)$,
and also $P_{18}$ and $P_{19}$ by any of $\{(-19, 4), (321, -74)\}$.
We choose the large $(991, 227)$ and the smaller $(-19, 4)$ and come
up with Figure~\ref{real19-step9}.
\begin{figure}[hbt]
\begin{tikzpicture}[>=latex,xscale=400,yscale=1000]
\draw[->] (0, 0.351) -- (0, 0.358); 

\coordinate (P22) at (0.000000, 0.352893);
            \node at (0.000000, 0.352893) {$\bullet$};
\coordinate (P20) at (0.000000, 0.351107);
            \node at (0.000000, 0.351107) {$\bullet$};
\coordinate (P21) at (0.003889, 0.352000);
            \node at (0.003889, 0.352000) {$\bullet$};

\draw[fill=lightgray] (0.000000, 0.352893) -- (0.000000, 0.351107) --
(0.003889, 0.352000) -- (0.000000, 0.352893);

\node (P_c) at (0.0000, 0.3509) {$\circ$}; \node at (-0.001, 0.351) {$P_c$};

\draw[domain=0:0.02] plot (\x,
{1-sqrt(((\x+(-3))*(\x+(-3))-0.9942)/19)}) node[right] {$\quad B_{-3,
-1}^{-, +}$};

\draw[domain=0:0.02] plot (\x, {-1+sqrt(((\x+(-6))*(\x+(-6))-0.9942)/19)}) node[right] {$B_{-6, 1}^{+, +}$};
\node (P_{18}) at (0.0000, 0.3574) {$\bullet$}; \node at (-0.001, 0.3575) {$P_{18}$};
\node (P_{19}) at (0.0136, 0.3542) {$\bullet$}; \node at (0.0136, 0.355) {$P_{19}$};

\draw[domain=0:0.02,very thick] plot (\x, {0.3511074+0.2294156179*\x}) node[above] {$B_{991, 227}^{+, -}$};
\draw[domain=0:0.02,very thick] plot (\x, {0.3508771930 +
0.2294158500* \x)}) node[below] {$\quad B_{991, 227}^{+, +}$};
\node at (-0.00100, 0.3515) {$P_{20}$};

\draw[domain=0:0.01,very thick] plot (\x,
{-4+sqrt(((\x+(-19))*(\x+(-19))-0.9942)/19)}) node[below] {$B_{-19,
4}^{+, +}$};
\node at (-0.001, 0.3529) {$P_{22}$};
\node at (0.0039, 0.3524) {$P_{21}$};
\end{tikzpicture}
\caption{The case $m=19$: step 9. \label{real19-step9}}
\end{figure}
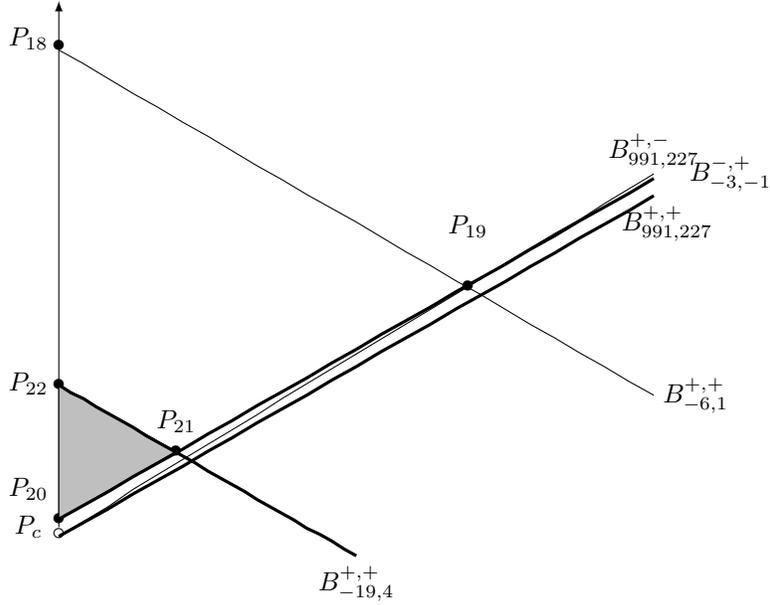
We get three new points
$$P_{20}=B_{991, 227}^{+-} \cap \{x=0\}=
\left(0,-227+{\frac {11\,\sqrt {1387901}}{57}}\approx 0.3511\right),$$
with the first part and
$$P_{21}=B_{-19, 4}^{++} \cap B_{991, 227}^{+-}=
\left(-{\frac{6253815094}{12867579}}+{\frac {\sqrt
{156443316403988655691}}{25735158}}\approx 0.0039\right.,$$
$$\quad\quad \left.-{\frac{2972486569}{25735158}}+{\frac
{505\,\sqrt {3145917199299979}}{244484001}}\approx 0.3520\right),$$
$$P_{22}=B_{-19, 4}^{++} \cap \{x=0\}=
\left(0,-4+{\frac {\sqrt {61561}}{57}}\approx 0.3529\right)$$
in the second.
First, we have to prove
$$B_{991, 227}^{+-} \geq B_{-3, -1}^{-+}\text{ for } a \in [0, x(P_{19})], \text{i.e.}, $$
$$-29241\,{a}^{2}-29060082\,a+480767 \geq 0: \text{case b};$$ 
$$B_{-3, -1}^{-+} \geq B_{991, 227}^{++}\text{ for } a \in [0, x(P_{19})], \text{i.e.}, 
{a}^{2}+988\,a \geq 0: \text{case }a_2.$$ 
Second, we need
$$B_{-19, 4}^{+-} \geq B_{-6, 1}^{++}\text{ for } a \in [0, 2/5], \text{i.e.}, 
-58482\,{a}^{2}+706230\,a+2838943 \geq 0: \text{case b};$$ 
$$B_{-6, 1}^{++} \geq B_{-19, 4}^{++}\text{ for } a \in [0, 2/5], \text{i.e.}, 
2\,{a}^{2}-50\,a+57 \geq 0: \text{case }a_1.$$ 

\medskip
\noindent
{\em Step 10:}
we are left with covering
$\overline{P_{22}P_{20}}\widehat{P_{20}P_{21}P_{22}}$.
We cover the pair $(P_{21}, P_{22})$ by $(-80, 18)$, then
$(P_{20}, P_{24})$ by $(-430, -99)$ to get Figure~\ref{real19-step11}.

\begin{figure}[hbt]
\begin{tikzpicture}[>=latex,xscale=400,yscale=1200]
\clip (-0.0025, 0.35) rectangle (0.014, 0.354);
\draw[->] (0, 0.35) -- (0, 0.354); 

\draw[domain=0:0.01] plot (\x, {0.3511074+0.2294156179*\x})
node[above] {$\quad\quad\quad\quad B_{991, 227}^{+, -}$};
\node (P_{20}) at (0.0000, 0.3511) {$\bullet$}; \node at (-0.00100, 0.351) {$P_{20}$};

\draw[domain=0:0.005] plot (\x,
{-4+sqrt(((\x+(-19))*(\x+(-19))-0.9942)/19)}) node[right] {$B_{-19,
4}^{+, +}$};
\node (P_{22}) at (0.0000, 0.3529) {$\bullet$}; \node at (-0.001, 0.3529) {$P_{22}$};
\node (P_{21}) at (0.0039, 0.3520) {$\bullet$}; \node at (0.0039, 0.3515) {$P_{21}$};
\draw[domain=0:0.01,very thick] plot (\x, {-18+sqrt(((\x+(-80))*(\x+(-80))+0.9942)/19)}) node[right] {$B_{-80, 18}^{+, -}$};
\draw[domain=0:0.005,very thick] plot (\x, {-18+sqrt(((\x+(-80))*(\x+(-80))-0.9942)/19)}) node[right] {$B_{-80, 18}^{+, +}$};

\draw[domain=0:0.01,very thick] plot (\x, {0.35096922+0.2294151172*\x}) node[right] {$B_{-430, -99}^{-, -}$};
\draw[domain=0:0.01,very thick] plot (\x, {0.35149963+0.2294163506*\x}) node[left] {$B_{-430, -99}^{-, +}\quad$};

\node (P_{23}) at (0.0000, 0.3518) {$\bullet$}; \node at (-0.001, 0.3518) {$P_{23}$};
\coordinate (P_{24}) at (0.001582, 0.351470);
       \node at (0.001582, 0.351470) {$\bullet$}; 
       \node at (0.001582, 0.351) {$P_{24}$};
\node (P_{26}) at (0.0000, 0.3515) {$\bullet$}; \node at (-0.001, 0.3515) {$P_{26}$};
\node (P_{25}) at (0.0007, 0.3517) {$\bullet$}; \node at (0.0007, 0.352) {$P_{25}$};

\end{tikzpicture}
\begin{tikzpicture}[>=latex,xscale=500,yscale=1500]
\clip (-0.003, 0.35) rectangle (0.0075, 0.354);
\draw[->] (0, 0.35) -- (0, 0.354); 

\draw[domain=0:0.004] plot (\x, {-18+sqrt(((\x+(-80))*(\x+(-80))-0.9942)/19)}) node[right] {$B_{-80, 18}^{+, +}$};

\draw[domain=0:0.004] plot (\x, {0.35149963+0.2294163506*\x}) node[right] {$B_{-430, -99}^{-, +}\quad$};

\node (P_{23}) at (0.0000, 0.3518) {$\bullet$}; \node at (-0.001, 0.3518) {$P_{23}$};
\node (P_{26}) at (0.0000, 0.3515) {$\bullet$}; \node at (-0.001, 0.3515) {$P_{26}$};
\node (P_{25}) at (0.0007, 0.3517) {$\bullet$}; \node at (0.0007, 0.352) {$P_{25}$};

\draw[domain=0:0.004,very thick] plot (\x,
{0.35131691-0.2294016565*\x}) node[right] {$B_{90, -21}^{-, -}$};
\draw[domain=0:0.004,very thick] plot (\x, {0.35385107-0.2294298138*\x})
node[right] {$B_{90, -21}^{-, +}$};

\end{tikzpicture}
\caption{The case $m=19$: step 10 (left), step 11
(right). \label{real19-step11}}
\end{figure}
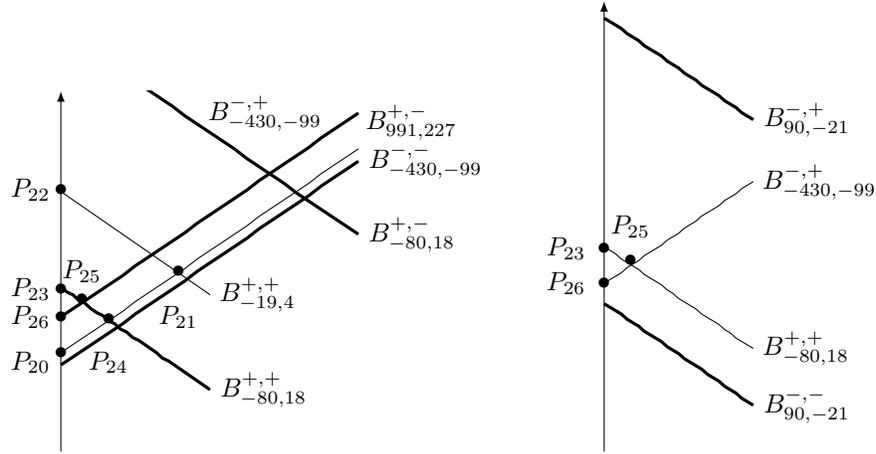
We get four new points
$$P_{23} = B_{-80, 18}^{++} \cap \{x=0\}=
\left(0, -18+{\frac {\sqrt {1094230}}{57}}\approx 0.3518\right),$$
$$P_{24}=B_{-80, 18}^{++} \cap B_{991,
227}^{+-}=\left(-{\frac{2744379489}{6024938}}+{\frac {11\,\sqrt
{13966378024053079}}{2853918}}\approx 0.001582,\right.$$
$$\left.-{\frac{349606825}{2853918}}+{\frac {119\,\sqrt {13966378024053079}}{114473822}}\approx 0.351470\right),$$
$$P_{25} = B_{-430, -99}^{-+} \cap B_{-80, 18}^{++}=
\left(255-{\frac {39\,\sqrt
{3237343002931}}{275182}}\approx 0.0007,\right.$$
$$\left.{\frac{81}{2}}-{\frac {175\,\sqrt
{3237343002931}}{7842687}}\approx 0.3517\right),$$
$$P_{26} = B_{-430, -99}^{-+} \cap \{x=0\}=
\left(0, 99-{\frac {\sqrt {31617730}}{57}}\approx 0.3515\right).$$
First, we need to prove
$$B_{-80, 18}^{+-} \geq B_{-19, 4}^{++}\text{ for } a \in [0, 2/5], \text{i.e.}, $$
$$-350892\,{a}^{2}+20552148\,a+169237849 \geq 0: \text{case b};$$ 
$$B_{-19, 4}^{++} \geq B_{-80, 18}^{++}\text{ for } a \in [0, 2/5], \text{i.e.}, 
108\,{a}^{2}-10692\,a+30799 \geq 0: \text{case }a_1.$$ 
Second:
$$B_{-430, -99}^{-+} \geq B_{991, 227}^{+-}\text{ for } a \in [0, 2/5], \text{i.e.}, $$
$$-350892\,{a}^{2}+133617348\,a+242222193409 \geq 0: \text{case b};$$ 
$$B_{991, 227}^{+-} \geq B_{-430, -99}^{--}\text{ for } a \in [0, 2/5], \text{i.e.}, 
108\,{a}^{2}+60588\,a+26245559 \geq 0: \text{case -}.$$ 

\medskip
\noindent
{\em Step 11:} the last (!) step is to cover the remaining three
points $P_{23}$, $P_{26}$, $P_{25}$. It turns out that
\begin{verbatim}
C[23]={[-90, -21], [-80, 18], [80, 18], [90, -21]}
C[24]={[-90, -21], [90, -21]}
C[25]={[-90, -21], [-80, 18], [90, -21]}
Cinter={[-90, -21], [90, -21]}
\end{verbatim}
and so we can use $(90, -21)$ for this task. We get
Figure~\ref{real19-step11}. Our last proofs for this case are
$$B_{90, -21}^{-+} \geq B_{-80, 18}^{++}\text{ for } a \in [0, 2/5], \text{i.e.}, 
4\,{a}^{2}+40\,a+86121 \geq 0: \text{case -};$$ 
$$B_{-80, 18}^{++} \geq B_{90, -21}^{--}\text{ for } a \in [0, 2/5], \text{i.e.}, $$
$$-116964\,{a}^{2}-40704840\,a+644319959 \geq 0: \text{case b}.$$ 
Finally
$$B_{90, -21}^{-+} \geq B_{-430, -99}^{-+}\text{ for } a \in [0, x(P_{25})], \text{i.e.}, 
229\,{a}^{2}-77860\,a+399 \geq 0: \text{case }a_1;$$ 
$$B_{-430, -99}^{-+} \geq B_{90, -21}^{--}\text{ for } a \in [0, 2/5], \text{i.e.}, $$
$$-1131655941\,{a}^{2}+384755461740\,a+153230834 \geq 0: \text{case b}.$$ 

\section{Computing the euclidean division}

We are given a number $\xi = a+b \omega$ with $0 \leq a, b \leq 1/2$
(coming from a centered division on two quadratic numbers) and we
want to find two rational integers $x$ and $y$ such $\gamma = x + y
\omega$ for which $|\Norm(\xi-\gamma)| \leq M$.
There are two possible algorithms. The first one is to find a
covering region for $(a, b)$ in our collection of regions. From an
algorithmic point of view, this is a bit cumbersome. The second
algorithm considers in sequence all possible coverings $(u, v)$
from our list, and stops as soon as
$|f_m(a+u, b+v)| \leq M$. (There might several possible
pairs $(u, v)$, one is enough for our needs, unless we insist on having
the smallest possible value for the norm.)

\section{Future work}

The first task is to extend this work to the case $m \equiv 1\bmod 4$,
a task which already began, highlighting easy cases ($5$, $17$, $21$,
$57$) but also harder cases like $13$ and $29$, see the forthcoming
\cite{Morain2026c}. We can try to use our approach to the case of
$M_2$-division algorithms where $M_2$ is the second Euclidean minimum,
and why not the cases of other $M_i$'s that are known in certain
cases. This could be much harder. In a more general context, there is
room for $M_1$-algorithms in larger degree, with possibly more
difficult work on the known cases.

\iffalse
\bibliographystyle{plain}
\bibliography{morain,euclidean}

\begin{thebibliography}{10}

\bibitem{BaSw1952b}
E.~S. Barnes and H.~P.~F. Swinnerton-Dyer.
\newblock The inhomogeneous minima of binary quadratic forms. {II}.
\newblock {\em Acta Math.}, 88:279--316, 1952.

\bibitem{Berg1935}
E.~Berg.
\newblock {\"U}ber die {E}xistenz eines {E}uklidischen {A}lgorithmus in
  quadratischen {Z}ahlk{\"o}rpern.
\newblock {\em Kungl. Fysiografiska sällskapets i Lund förhandlingar},
  5(5):53--58, 1935.

\bibitem{CaSc2010}
Perlas~C. Caranay and Renate Scheidler.
\newblock An efficient seventh power residue symbol algorithm.
\newblock {\em Int. J. Number Theory}, 6(8):1831--1853, 2010.

\bibitem{Cerri2006}
Jean-Paul Cerri.
\newblock Inhomogeneous and {E}uclidean spectra of number fields with unit rank
  strictly greater than 1.
\newblock {\em J. Reine Angew. Math.}, 592:49--62, 2006.

\bibitem{Cerri2007}
Jean-Paul Cerri.
\newblock Euclidean minima of totally real number fields: algorithmic
  determination.
\newblock {\em Math. Comp.}, 76(259):1547--1575, 2007.

\bibitem{Dirichlet1893}
P.~G.~Lejeune Dirichlet.
\newblock {\em Vorlesungen {\"u}ber {Z}ahlentheorie}.
\newblock Vieweg, Braunschweig, 1893.
\newblock ed. R. Dedekind.

\bibitem{JoLaNgNa2021}
Marc Joye, Oleksandra Lapiha, Ky~Nguyen, and David Naccache.
\newblock The eleventh power residue symbol.
\newblock {\em J. Math. Cryptol.}, 15(1):111--122, 2021.

\bibitem{Lemmermeyer1995}
Franz Lemmermeyer.
\newblock The {E}uclidean algorithm in algebraic number fields.
\newblock {\em Exposition. Math.}, 13(5):385--416, 1995.
\newblock Updated version, 2004.

\bibitem{Lezowski2014}
Pierre Lezowski.
\newblock Computation of the {E}uclidean minimum of algebraic number fields.
\newblock {\em Math. Comp.}, 83(287):1397--1426, 2014.

\bibitem{Morain2026a}
F.~Morain.
\newblock Division algorithms for {E}uclidean imaginary quadratic fields.
\newblock Preprint, January 2026.

\bibitem{Morain2026c}
F.~Morain.
\newblock Division algorithms for norm-{E}uclidean real quadratic fields --
  part {II}.
\newblock In preparation, January 2026.

\bibitem{Oppenheim1934}
Alexander Oppenheim.
\newblock Quadratic fields with and without {E}uclid's algorithm.
\newblock {\em Math. Ann.}, 109(1):349--352, 1934.

\bibitem{Perron1933}
Oskar Perron.
\newblock Quadratische {Z}ahlk\"orper mit {E}uklidischem {A}lgorithmus.
\newblock {\em Math. Ann.}, 107(1):489--495, 1933.

\bibitem{Remak1934}
Robert Remak.
\newblock Über den {E}uklidischen {A}lgorithmus in reell-quadratischen
  {Z}ahlkörpern.
\newblock {\em Jahresbericht der Deutschen Mathematiker-Vereinigung},
  44:238--250, 1934.

\bibitem{ScWi1995}
Renate Scheidler and Hugh~C. Williams.
\newblock A public-key cryptosystem utilizing cyclotomic fields.
\newblock {\em Des. Codes Cryptogr.}, 6(2):117--131, 1995.

\bibitem{Varnavides1948ab}
P.~Varnavides.
\newblock Non-homogeneous binary quadratic forms. {I}, {II}.
\newblock {\em Nederl. Akad. Wetensch., Proc.}, 51:396--404, 470--481 =
  Indagationes Math. 10, 142--150, 164--175, 1948.
\newblock case Q(sqrt(2)).

\bibitem{Varnavides1949}
P.~Varnavides.
\newblock On the quadratic form {$x^2-7y^2$}.
\newblock {\em Proc. Roy. Soc. London Ser. A}, 197:256--268, 1949.

\bibitem{Varnavides1952}
P.~Varnavides.
\newblock The {M}inkowski constant of the form {$x^2-11y^2$}.
\newblock {\em Bull. Soc. Math. Gr\`ece}, 26:14--23, 1952.

\end{thebibliography}
\else
\def\noopsort#1{}\ifx\bibfrench\undefined\def\biling#1#2{#1}\else\def\biling#1#2{#2}\fi\def\Inpreparation{\biling{In
  preparation}{en
  pr{\'e}paration}}\def\Preprint{\biling{Preprint}{pr{\'e}version}}\def\Draft{\biling{Draft}{Manuscrit}}\def\Toappear{\biling{To
  appear}{\`A para\^\i tre}}\def\Inpress{\biling{In press}{Sous
  presse}}\def\Seealso{\biling{See also}{Voir
  {\'e}galement}}\def\Editor{\biling{Ed.}{R{\'e}d.}}

\fi

\end{document}